\documentclass[11pt,oneside,reqno]{amsart}  

\usepackage[utf8]{inputenc}            
\usepackage[T1]{fontenc}
\usepackage{microtype}

\usepackage[english]{babel}  
    
\usepackage{url}
\usepackage{xcolor}                   
\usepackage{graphics,graphicx,xypic}
\usepackage{caption}
\usepackage{subcaption}
\usepackage{amsthm,amscd}                     
\usepackage{dsfont,txfonts,bbold} 
\usepackage{mathrsfs}
\usepackage{textcomp}
\usepackage{enumerate} 
\usepackage{hyperref}
\usepackage[toc,page]{appendix}         

\usepackage{natbib}

\usepackage[a4paper,scale={0.72,0.74},marginratio={1:1},footskip=7mm,headsep=10mm]{geometry}

\numberwithin{equation}{section}    

\newtheorem{theorem}{Theorem}[section]

\newtheorem{proposition}[theorem]{Proposition}

\newtheorem{lemma}[theorem]{Lemma}

\theoremstyle{definition}
\newtheorem{definition}[theorem]{Definition}
\newtheorem{remark}[theorem]{Remark}

\newcommand{\N}{\mathbb{N}}

\newcommand{\ind}{\mathds{1}}

\newcommand{\tx}[1]{\texttt{#1}}
\newcommand{\vl}{\, \tilde{\wedge}\, }

\title{Collapse transition of the interacting prudent walk}

\author{Nicolas Pétrélis}
\address{Laboratoire de Mathématiques Jean Leray UMR 6629, 2, rue de la Houssinière
BP 92208, 44322 Nantes Cedex 3}

\email{nicolas.petrelis@univ-nantes.fr}
\author{Niccolò Torri}
\address{Laboratoire de Mathématiques Jean Leray UMR 6629, 2, rue de la Houssinière
BP 92208, 44322 Nantes Cedex 3}
\email{niccolo.torri@univ-nantes.fr}
\thanks{N.T. was supported by the \emph{ ``Investissements d'avenir"} program (ANR-11-LABX-0020-01).}

\newenvironment{myenumerate}{%
\renewcommand{\theenumi}{\arabic{enumi}}%
\renewcommand{\labelenumi}{{\rm(\theenumi)}}%
\begin{list}{\labelenumi}
	{%
	\setlength{\itemsep}{0.4em}%
	\setlength{\topsep}{0.5em}%
	\setlength\leftmargin{2.45em}%
	\setlength\labelwidth{2.05em}%
	\setlength{\labelsep}{0.4em}%
	\usecounter{enumi}%
	}%
	}%
{\end{list}
}

{\end{list}
}

\renewenvironment{enumerate}{
\begin{myenumerate}}%
{\end{myenumerate}}

\newenvironment{myitemize}{%
\begin{list}{$\bullet$}%
 	{%
	\setlength{\itemsep}{0.4em}%
	\setlength{\topsep}{0.5em}%
	\setlength\leftmargin{2.45em}%
	\setlength\labelwidth{2.05em}%
	\setlength{\labelsep}{0.4em}%
	}%
	}%
{\end{list}}

\renewenvironment{itemize}{
\begin{myitemize}}%
{\end{myitemize}}

\begin{document}

\begin{abstract}

This article is dedicated to the study of the 2-dimensional interacting prudent self-avoiding walk (referred to by the acronym IPSAW) and in particular to its collapse transition. The interaction intensity is denoted by $\beta>0$ and the set of trajectories  consists of those self-avoiding paths 
respecting the prudent condition, which means that they do  not take a step towards a previously visited lattice site.
The IPSAW interpolates between the interacting partially directed self-avoiding walk (IPDSAW) that was analyzed in details  in,  e.g., 
\cite{ZL68}, \cite{BGW92}, \cite{CGP13} and  \cite{GP13}, and the interacting self-avoiding walk (ISAW)  for which the collapse transition was conjectured in \cite{S86}.

Three main theorems are proven. We show first that  IPSAW undergoes a collapse transition at finite temperature and, up to our knowledge, there was so far no proof in the literature of the existence of a collapse transition for a non-directed model  
built with self-avoiding path. We also prove that the 
free energy of IPSAW is equal to that of a restricted version of IPSAW, i.e., the interacting two-sided prudent
walk. Such free energy is computed by considering only those prudent path with a general  north-east orientation. As a by-product of this result 
we obtain  that  the exponential growth rate of generic prudent paths equals that of two-sided prudent paths and this answers an open problem raised in  e.g., \cite{B10} or \cite{DG08}.   Finally we show that, for every $\beta>0$,  the free energy of ISAW itself is always larger than $\beta$ and this rules out a possible self-touching saturation of ISAW in its conjectured collapsed phase.

%
%
%

\end{abstract}

\maketitle

\section{Introduction}
The collapse transition of  self-interacting random walks is a challenging issue, arising in the study of the 
$\theta$-point of an homopolymer dipped in a repulsive solvent. Different mathematical models have been built by physicists 
to try and improve their understanding of this phenomenon. For such models, the possible 
spatial configurations 
of the polymer are provided by random walk trajectories. In \cite{S86}, Saleur studies the \emph{interacting self-avoiding walk} (referred to as ISAW) that is built with self-avoiding paths which are relevant from the physical viewpoint because they fulfill 
the exclusion volume effect, a feature that real-world polymers indeed satisfy. However, self-avoiding paths, especially in dimension $2$ and $3$, 
are complicated objects. This is the reason why, in the mathematical literature, collapse transition models were rather built by 
either relaxing the self-avoiding feature of the paths (see for instance \cite{vdHK01} or 
\cite{vdHKK02}) or by considering partially directed paths. 
This is the case for  the \emph{interacting partially directed self-avoiding walk} (referred to as IPDSAW) 
that was introduced in \cite{ZL68} and subsequently studied 
in e.g.  \cite{BGW92} or   \cite{GP13}, \cite{CGP13}
and \cite{CP15}).

In the present paper, we focus on the \emph{interacting prudent self-avoiding walk} (referred to as IPSAW), a model 
built with prudent paths, i.e., non-directed self-avoiding paths which can not take a step towards a previously visited lattice site. 
The IPSAW clearly interpolates between IPDSAW and ISAW since partially directed paths are prudent paths which themselves are self-avoiding paths. 
An interesting feature of prudent paths is that although they are non-directed and self-avoiding, the prudent condition, especially in dimension 2, imposes some geometric constraints that makes them more tractable than self-avoiding paths themselves.  This can be observed in the existing literature dedicated to prudent walks e.g., in  \cite{B10} or  \cite{BFV10}.


\subsection*{Organization of the paper}
 In Section \ref{the model}, we give a rigorous mathematical definition of IPSAW and we state our main results. Section \ref{bar} is dedicated to the 
comparison of our result with the existing literature. We will in particular show how IPSAW can be viewed as a limiting case of the undirected polymer in a poor solvent 
studied in \cite{vdHK01} and \cite{vdHKK02} and therefore shed some new light on the existence of a conjectured critical curve for this model. In Section \ref{decomp}, 
we start by increasing the complexity of the partially directed self-avoiding path by introducing the \emph {two-sided} prudent self-avoiding path.  Then, we show how to decompose a generic prudent path into a collection of two-sided paths.
Section \ref{proof of Thm1} is dedicated to the proof of Theorem \ref{Thm1} that states the existence of a  collapse transition for IPSAW
at finite temperature.  Section \ref{proof of thh1} provides an algorithm 
which shows that the free energy of IPSAW coincides with that of North-East interacting prudent self-avoiding walk (referred to as NE-IPSAW), which is a restriction of IPSAW built with a particular type of two-sided paths, i.e., the Nort-East prudent paths. With Section \ref{proofThm3}, we provide a lower bound on the free energy of ISAW which allows us to compare the nature of the collapse transitions of IPDSAW or IPSAW with that of 
ISAW. Finally, in Section \ref{app:freeenergy} we prove the existence of the free energy of NE-IPSAW.


\section{The interacting prudent self-avoiding walk (IPSAW)}\label{the model}

\subsection{Description of the models} 
Let $L\in \N$ be the system size and let $\Omega_L^{\tx{SAW}}$ be the set of $L$-step prudent paths in $\mathbb{Z}^2$, i.e.,
\begin{align}
\nonumber \Omega_L^{\tx{PSAW}}=\big\{w:=(w_i)_{i=0}^L\in (\mathbb{Z}^2)^{L+1}\colon\,& w_{0}=0,\  w_{i+1}-w_i\in \{\leftarrow,\rightarrow,\downarrow,\uparrow\},\  0\leq i \leq L-1,\\
&w \  \text{satisfies the prudent condition}\big\},
\end{align}
where the \emph{prudent condition} for a path $w$ means that it does not take any step in the direction 
of a lattice site already visited.
We also consider a subset of $\Omega_{L}^{\tx{PSAW}}$ denoted by $\Omega_L^{\tx{NE}}$ containing those $L$-step prudent paths 
with a general \emph{north-east} orientation. We postpone the precise definition of $\Omega_L^{\tx{NE}}$ to Section \ref{sec:NE-IPSAW} because this requires some additional notations  but one easily understands what such path look like with Figure \ref{figb}.

At this stage we build two polymer models: the IPSAW for which the  set of allowed spatial configurations for the polymer is given by $\Omega_L^{\tx{PSAW}}$ and its North-East counterpart (NE-IPSAW)
for which the  set of  configurations  is given by $\Omega_L^{\tx{NE}}$. For both models, each step of the walk is an abstract monomer and 
 we want to take into account the repulsion between monomers and the environment around them. 
This is achieved indirectly, by encouraging monomers to attract each other,
  i.e.,  by assigning an energetic reward $\beta\geq 0$  to any pair of non-consecutive 
  steps of the walk though adjacent on the lattice $\mathbb{Z}^2$. To that aim, 
we associate with every path $w$ the sequence of those points in the middle of each step, i.e., $u_i=w_{i-1}+\frac{w_i-w_{i-1}}{2}$ ($1\leq i\leq L$) and 
we reward every non-consecutive pair $(u_i,u_j)$ at distance one, i.e, $\|u_i-u_j\| = 1$, see Figure \ref{fig:IPDRW1}.  The energy associated with a given $w\in \Omega_L$ is defined by an explicit Hamiltonian, that is 
\begin{equation}\label{eq:ham1}
\mathrm H\,(w):=\sum_{\substack{i,j=0\\i<j}}^L\ind_{\{\lVert u_i-u_j\rVert=1\}},
\end{equation}
so that $\mathrm{Z}_{\beta, L}$ the partition function of IPSAW and $\mathrm{Z}^{\tx{NE}}_{\beta,\, L}$ the partition function of the North-east model equal
\begin{equation}
\mathrm{Z}_{\beta,\, L}:=\sum_{w\in \Omega_L^{\tx{PSAW}}} e^{\, \beta \,   \mathrm H\,(w)} \quad \text{and}  \quad \mathrm{Z}^{\tx{NE}}_{\beta,\, L}:=\sum_{w\in \Omega_L^{\tx{NE}}} e^{\, \beta \,   \mathrm H\, (w)}.
\end{equation}


\smallskip

The key objects of our analysis are the free energies of both models, i.e., $\mathrm F(\beta)$ and  $\mathrm F^{\tx{NE}}(\beta)$ which record the exponential growth rate of the partition function sequences $(\mathrm{Z}_{\beta,\, L})_{L\in \N}$ and $(\mathrm{Z}^{\tx{NE}}_{\beta,\, L})_{L\in \N}$, respectively.  Thus, 
\begin{equation}
\label{eq:FreeEnergyIPSAW}
\mathrm F(\beta):= \lim_{L\to\infty}\frac{1}{L}\log \mathrm Z_{\beta,\, L} \quad \text{and} \quad \mathrm F^{\tx{NE}}(\beta):= \lim_{L\to\infty}\frac{1}{L}\log \mathrm Z^{\tx{NE}}_{\beta,\, L}.
\end{equation}
The convergence in the right hand side of \eqref{eq:FreeEnergyIPSAW} will be proven in Section \ref{app:freeenergy}. The convergence in the l.h.s. of  \eqref{eq:FreeEnergyIPSAW} is more complicated and  it will   be 
obtained as a by-product of Theorem \ref{thh1} below.

\subsection{Main results}
\label{MainRes}
 
In the present Section we 
state our main results and we give some hints about their proof. We pursue the discussion in Section \ref{bar} below,  
by explaining how our results answer some open problems leading to a better  comprehension of interacting self-avoiding walk. 

\smallskip

 With Theorem \ref{thh1} below, 
we state that the free energies of  IPSAW and of NE-IPSAW are equal. Our proof  
is displayed in Section \ref{proof of thh1} and is purely combinatorial.  It consists in 
building a sequence of path transformations $(M_L)_{L\in \N}$ such that 
for every $L\in \N$, $M_L$ maps any generic path in $\Omega_L^{\tx{PSAW}}$ onto a \emph{2-sided} prudent path in $\Omega_L^{\tx{NE}}$ and  satisfies the following properties:
\begin{itemize}
\item for every $w\in \Omega_L^{\tx{PSAW}}$,  the difference between the Hamiltonians of $w$ and of $M_L(w)$ is $o(L)$,
\item the number of ancestors of a given path in $\Omega_L^{\tx{NE}}$ by $M_L$ can be shown to be $e^{o(L)}$.
\end{itemize} 
%
Such a mapping allows us to prove the following theorem. 
\begin{theorem}\label{thh1} For $\beta\geq 0$, 
\begin{equation}\label{thh}
\mathrm  F(\beta)= \mathrm  F^{\rm\tx{NE}}(\beta).
\end{equation}
\end{theorem}

The free energy equality in \eqref{thh} will subsequently be used to establish Theorem \ref{Thm1} below, which states that IPSAW undergoes a collapse transition at finite temperature. 

\medskip

  \begin{theorem}
\label{Thm1}
There exists  a $\beta_c^{\tx{IPSAW}}\in (0,\infty)$ such that 
\begin{align}
\mathrm  F(\beta)&>\beta \quad  \text{for every} \quad  \beta<\beta_c^{\tx{IPSAW}},\\
\nonumber\mathrm   F(\beta)&=\beta \quad  \text{for every} \quad  \beta\geq \beta_c^{\tx{IPSAW}}.
\end{align} 
Thus, the phase diagram $[0,\infty)$ is partitioned into 
a \emph{collapsed phase}, $\mathcal{C}:= [\beta_c^{\tx{IPSAW}},\infty)$ inside which the free energy \eqref{eq:FreeEnergyIPSAW} 
 is linear and an \emph{extended phase}, $\mathcal{E}=[0,\beta_c^{\tx{IPSAW}})$. 
\end{theorem}
The proof of Theorem \ref{Thm1} is displayed in Section \ref{proof of Thm1}.  It  requires to exhibit a loss of 
analyticity of $\beta \mapsto \mathrm F(\beta)$ at some positive value of $\beta$ (which is subsequently denoted by $\beta_c^{\tx{IPSAW}})$. 
The nature of the proof is much more probabilistic than that of Theorem \ref{thh1}. It indeed relies, on the one hand, on the random walk representation of the partially directed version 
of our model displayed initially in  \cite{GP13}  and, on the other hand, on the fact that 
 prudent path can be naturally decomposed into shorter partially directed paths.

Since a partially directed self-avoiding path is in particular a generic prudent path, we 
can compare the critical point  of IPSAW with the critical point of IPDSAW, which was computed explicitly in e.g. \cite{BGW92, GP13}.
We obtain that 
\begin{equation}
\label{eq:IPIPDcriticalpoint}
\beta_c^{\tx{IPDSAW}} \, \leq\, \beta_c^{\tx{IPSAW}}.
\end{equation}
The inequality in \eqref{eq:IPIPDcriticalpoint} is somehow not 
satisfactory since one wonders whether it is strict or not. This issue is left as an open question and will be discussed further in Section \ref{Open problems}. 

We conclude this section by considering the $2$-dimensional Interacting Self-Avoiding Walk (ISAW) defined exactly like the IPSAW in (\ref{eq:ham1}) but with a larger set of allowed configurations, that is (in size $L\in \N$)
\begin{align}\label{allconf}
\Omega_L^{\tx{SAW}}:=\big\{w:=(w_i)_{i=0}^L\in (\mathbb{Z}^2)^{L+1}\colon\,& w_{0}=0,\  w_{i+1}-w_i\in \{\leftarrow,\rightarrow,\downarrow,\uparrow\},\  0\leq i \leq L-1,\\
\nonumber &w \  \text{satisfies the self-avoiding condition}\big\}.
\end{align}
We denote by $Z_{L,\beta}^{\tx{ISAW}}$ the partition function of ISAW  and we define its  free energy as
\begin{equation}\label{defISAW}
F^{\tx{ISAW}}(\beta):=\liminf_{L\to \infty} \frac{1}{L} \log Z_{L,\beta}^{\tx{ISAW}},
\end{equation}
where the $\liminf$ in \eqref{defISAW} is chosen to overstep the fact that 
the convergence of the free energy remains an open issue.
\begin{theorem}
\label{Thm3}
\begin{align}
\mathrm  F^{\tx{ISAW}}(\beta)&>\beta, \quad  \text{for every} \quad  \beta\in [0,\infty).
\end{align} 
\end{theorem}
A straightforward consequence of Theorem \ref{Thm3} is that  the conjectured collapse transition displayed by ISAW at some 
$\beta_c^{\tx{ISAW}}$ does 
not correspond to a self-touching saturation as it is the case for 
IPDSAW and IPSAW. 

\section{Discussion} \label{bar}

\subsection{Background}

The ISAW has triggered quite a lot of attention from both the physical and the mathematical communities. Much efforts have been put, for instance,   to 
 estimate numerically the value of the critical point $\beta_c^{\tx{ISAW}}$ (see e.g. \cite{TROW96a} or \cite{TROW96b} in dimension $3$) or to compute the typical end to end distance of a  
 path at criticality (see e.g. \cite{S86}). However, 
only very few rigorous mathematical results have been obtained about it so far. For example, the existence of a collapse transition is conjectured only 
and if such transition turns out to occur, obtaining some quantitative results about the geometric conformation adopted by the path inside each phase 
is even more challenging. In view of the mathematical complexity of ISAW, other models have been introduced, somehow simpler than ISAW
and therefore more tractable mathematically.

The first attempt to investigate a simplified version of ISAW is due to  \cite{ZL68} with the Interacting Partially Directed Self-Avoiding Walk
 (IPDSAW).
Again the model is defined as in   (\ref{eq:ham1}), but with a restricted set of configurations, i.e.,
\begin{align}\label{setofIPDSAW}
\Omega_L^{\tx{PDSAW}}:=\big\{w:=(w_i)_{i=0}^L\in (\mathbb{N}\times \mathbb{Z})^{L+1}\colon\,& w_{0}=0,\  w_{i+1}-w_i\in \{\rightarrow,\downarrow,\uparrow\},\  0\leq i \leq L-1,\\
\nonumber &w \  \text{satisfies the self-avoiding condition}\big\}.
\end{align}
The IPDSAW was first investigated with combinatorial methods in e.g., \cite{BGW92} 
where the critical temperature, $\beta_c^{\tx{IPDSAW}}$, is computed. Subsequently, 
%
%
in \cite{GP13}  and  \cite{CGP13} and \cite{CP15}  a probabilistic
 approach allowed for a rather complete quantitative description of the scaling limits displayed by 
 IPDSAW in each three regimes (extended, critical and collapsed).

%
%
%
Another simplification of ISAW gave birth to the Interacting Weakly Self-Avoiding Walk (IWSAW), which  is built by  relaxing the self-avoiding 
condition imposed in ISAW such that the set of configurations $\Omega_L$
contains every $L$-step trajectory of a discrete time simple random walk on $\mathbb Z^d$ ($d\geq 1$). 
The Hamiltonian associated with every path rewards the self-touchings and penalizes the self-intersections, i.e, 
for every $w\in \Omega_L$, 
\begin{equation}\label{hamil}
H\, (w)=-\gamma \sum_{0\leq i<j\leq L} \ind_{\{w_i-w_j=0\}}\ + \beta  \sum_{0\leq i<j\leq L} \ind_{\{|u_i-u_j|=1\}}.
\end{equation}
Thus, $\gamma\geq 0$
is a parameter that can be tuned to approach the ISAW through the IWSAW, since in the limit $\gamma=\infty$ both models coincide.
The IWSAW  is investigated in two papers, i.e.,  \cite{vdHK01} and  \cite{vdHKK02}  whose results are reviewed in  \cite[Section 6.1]{dH07}.
%
%
In  \cite{vdHKK02}, the existence of a critical curve $\gamma=2d \beta$ between a localized phase 
and a collapsed phase (also referred to as minimally extended) is proven in every dimension 
$d\geq 1$. Inside the localized phase (i.e., for $\beta>\gamma/2d$) and with probability arbitrarily close to $1$ 
the polymer is confined inside a squared box of finite size. Inside the collapsed phase in turn, the typical diameter  of the polymer 
is proven to be at least $L^{1/d}$. It is conjectured that at criticality ($\beta=2d \gamma$), the polymer scales as $L^{1/d+1}$. This is made rigorous in  \cite{vdHKK02}  when  $d=1$. 
In dimension $d\geq 2$, IWSAW is conjectured to undergo another critical curve $\gamma\mapsto \beta(\gamma)$ 
between the previously mentioned collapsed phase and an extended phase inside which the typical extension of the path is expected to be that of the self-avoiding walk. This
critical curve is expected to have an horizontal asymptote $\beta=\beta^*\in (0,\infty)$ and $\beta^*$ is itself expected to equal $\beta_c^{\tx{ISAW}}$.

\subsection{Discussion of the results}\label{bar2}
As mentioned above, one of the interest of IPSAW is that it interpolates between IPDSAW,  which is now very well understood, and ISAW (or IWSAW at $\gamma=\infty$) about which most theoretical issues remain open.  
From this perspective, Theorem \ref{Thm1} clearly constitutes a step forward in the 
investigation of ISAW since, up to our knowledge, IPSAW is the first non-directed model of interacting self-avoiding walk for which the existence 
of a collapse transition is proven rigorously.

At first sight, Theorem \ref{thh1} may appear as an intermediate step in the proof of theorem \ref{Thm1}.   The fact that 
the free energies of IPSAW and of NE-IPSAW are equal allows us 
to prove  Theorem \ref{Thm1} with $2$-sided prudent paths only. However, 
the importance of Theorem \ref{thh1}  goes beyond IPSAW  itself.
The \emph{2-sided} prudent trajectories have indeed been  studied already  in 
the mathematical litterature, see e.g., \cite{B10},   \cite{DG08} or   \cite{BI15}. 
It was conjectured in \cite{B10} or \cite{DG08} that the exponential growth rate of the cardinality of $2$-sided prudent paths (as a function of their length) equals that of generic prudent paths and this is precisely what Theorem \ref{thh1} says at $\beta=0$.
Moreover this result supports somehow the conjecture that the scaling limit of the uniform prudent walk should be the same as that of its
$2$-sided counterpart, see \cite{B10}. We discuss this conjecture in Section \ref{Open problems} below.

As mentioned below Theorem \ref{Thm3},
the fact that  ISAW does not give rise to a self-touching saturation when $\beta$ becomes large enough indicates that the nature 
of its phase transition differs from that of IPDSAW and IPSAW. 
Theorem \ref{Thm3} tells us that 
for every $\beta>0$,  one can display a subset of trajectories whose contribution to the  free energy  is strictly larger than $\beta$.
As a consequence, there is no straightforward inequality between the conjectured critical point $\beta_c^{\text{ISAW}}$  and  $\beta_c^{\text{IPDSAW}}$ or between $\beta_c^{\text{ISAW}}$ and $\beta_c^{\text{IPSAW}}$.


\smallskip

\subsection{Open problems}\label{Open problems}
We state 3 open problems which, in our opinion, are interesting but require to bring the
instigation of IPSAW and ISAW some steps further. We discuss those 3 issues subsequently.   
\begin{enumerate}
\item Compute $\beta_c^{\tx{IPSAW}}$ and therefore determine whether or not
 $\beta_c^{\tx{IPSAW}}>\beta_c^{\tx{IPDSAW}}$.
\item Provide the scaling limit of IPSAW in its three regimes, i.e., extended, critical and  collapsed. 
\item Prove that ISAW also undergoes a collapse transition at some $\beta>0$.
\end{enumerate}
Concerning the  first open question above, one should keep in mind Theorem \ref{thh1}. Proving that $\beta_c^{\tx{IPSAW}}>\beta_c^{\tx{IPDSAW}}$ indeed requires to check  that $\mathrm F^{\tx{NE}}(\beta_c^{\tx{IPDSAW}})>\beta_c^{\tx{IPDSAW}}$.
For simplicity we set  $\beta_c=\beta_c^{\tx{IPDSAW}}$. We recall the grand canonical characterization of the free energy, i.e.,
\begin{equation}\label{GC}
\mathrm F^{\tx{NE}}(\beta_c)-\beta_c=\inf\bigg\{\gamma>0\colon\;  \sum_{L\geq 1} Z_{L,\beta_c}^{\tx{NE}} e^{-(\beta_c+\gamma) L} <\infty\bigg\}
\end{equation}
and we observe that a generic NE-prudent path is a concatenation of partially directed path 
(see \eqref{defNE}) satisfying an additional geometric constraint called  \emph{exit-condition} (see Definition \ref{def:exit}). 
If we denote by $Z_{L,\beta_c}^{\tx{IPDSAW}}(*)$ the partition function of IPDSAW restricted to those configurations respecting the
\emph{exit-condition} and  if  we forget about the interactions between the partially directed subpaths constituting a NE-prudent path, we deduce that the inequality 
\begin{equation}\label{ineql}
\sum_{L\geq 4} Z_{L,\beta_c}^{\tx{IPDSAW}}(*)\,  e^{-\beta_cL} >1
\end{equation}
would be sufficient to claim that the l.h.s. in \eqref{GC} is positive. 
 Without the \emph{exit condition}, i.e., with  $Z_{L,\beta_c}^{\tx{IPDSAW}}$
instead of its restricted counterpart, the inequality \eqref{ineql} is true. This is a consequence 
of the random walk representation of IPDSAW displayed in 
\cite{GP13} which gives that  $\sum_{L\geq 4} Z_{L,\beta_c}^{\tx{IPDSAW}}\,  e^{-\beta_cL}=\infty$
because it equals the expected number of visits at the origin of a recurrent random walk on $\mathbb{Z}$. However, the  \emph{exit condition}
 imposed to every partially directed subpath constituting a NE-prudent path induces a strong loss of entropy and this is why
 we are not able to show that \eqref{ineql} also holds true.

 \smallskip

The second open question would complete the scaling limit of the prudent walk (at $\beta=0$). 
This problem has been investigated with combinatorial technics in, e.g.  \cite[Proposition 8]{B10} for the \emph{3-sided} prudent walk. In this case the scaling limit is a
straight line along the diagonal 
and it is conjectured that also the generic prudent walk displays the same
scaling limit. 
With probabilistic tools, the scaling limit of the (\emph{kinetic}) prudent walk was explored in  \cite{BFV10}.
We refer to \cite{BFV10} for the precise definition of the kinetic prudent walk, but let us 
emphasize that its scaling limit  is described by an explicit non trivial 
continuous process, cf. \cite[Theorem 1]{BFV10}.
 
 We may assume that inside its extended phase the scaling limit of 
IPSAW remains very similar to that of the prudent walk (at $\beta=0$). 
 From this perspective, it would be interesting to  get a better understanding of the geometry of IPSAW inside its collapsed phase as well. 
Since $F(\beta)=\beta$ when $\beta\geq \beta_c^{\text{IPSAW}}$, we can state that  the fraction of self-touching of a typical path is $1+o(1)$.
However,  there are various type of paths 
achieving 
this condition, e.g., the collapsed configurations of IPDSAW  (see \cite[Section 4]{CGP13}) or configurations filling a square box by turning around their range, and it is not clear at this stage which subclass would 
contribute the most to the partition function.

\smallskip

The third open question is  the most difficult.  The fact that one can not display a subset of parameters in $[0,\infty)$ inside which 
the free energy of ISAW becomes linear illustrates this  difficulty.  


%
%
 \begin{figure}[ht]
   \begin{subfigure}[t]{0.3\textwidth}
{\includegraphics[scale=1]{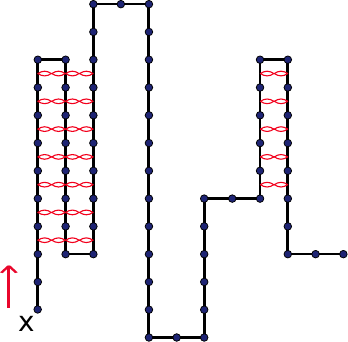}}
\caption{IPDSAW}\label{figa}
\end{subfigure}
 \begin{subfigure}[t]{0.3\textwidth}
{\includegraphics[scale=1]{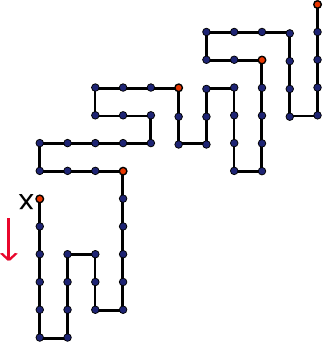}}
\caption{NE-IPSAW}\label{figb}
\end{subfigure}
 \begin{subfigure}[t]{0.3\textwidth}
{\includegraphics[scale=1.0]{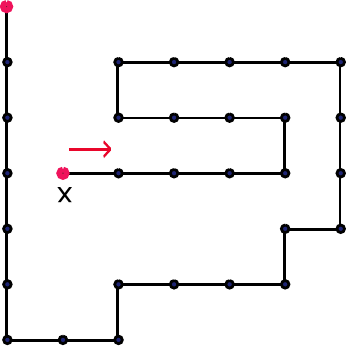}}
\caption{IPSAW}\label{figc}
\end{subfigure}
\caption{
Examples of  a PDSAW ($\text A$), NE-PSAW ($\text B$) and PSAW ($\text C$) path. 
Any path starts at $x$ and its orientation is 
given by the arrow.\\
In ($\text A$) we have drawn an IPDSAW path  
made of $11$ stretches: $\ell_1=9,\, \ell_2=-7,\, \ell_3=9
,\, \ell_4=0,\, \ell_5=-12,\, \ell_6=0,\, \ell_7 = 5,\, \ell_8=0,\, \ell_9=5, \,
\ell_{10}=-7, \ell_{11}=0$. That path performs 
 $19$ self-touching (drawn in red).}
\label{fig:IPDRW1}
\end{figure} 
%
%

%
%
%
%

\section{Decomposition of a generic prudent path}\label{decomp}

In this section we describe the different type of path that we will have to take into account in the paper. By order of increasing complexity, we will 
first introduce in Section \ref{sec:IPDSAW} the partially directed self-avoiding paths and their counterparts satisfying the 
so called \emph{exit condition} which is  an additional geometric constraints  allowing for their concatenation.
In section \ref{sec:NE-IPSAW}, we concatenate such partially directed paths to build the \emph{two-sided} prudent paths.
Those two sided path have $4$ possible general orientations; north-east (NE), north-west (NW), south-east (SE) and south-west (SW).
Finally in Section \ref{sec:IPSAWdef}, we will introduce the generic prudent path and observe that 
each such path can be decomposed in a unique manner into a succession of \emph{macro-blocks} that are particle case of two-sided prudent paths obeying some additional constraints given by the prudent condition to make possible their concatenation.

\smallskip

We need to define $\oplus$ a concatenation operator on prudent path. We  pick $r\in \N$ and we consider $r$  prudent paths denoted by $w_1,\dots,w_r$.  We 
let $w_1\oplus w_2\oplus\dots\oplus w_r$ be the path obtained by attaching the last step of $w_{i-1}$ with the first step of $w_i$ for every $2\leq i\leq r$. Then, the sequence $(w_1,\dots,w_r)$ is said to be \emph{concatenable} if 
$w_1\oplus \dots\oplus w_r$ itself is a prudent path. Finally, we extend the notation $\oplus$ to the concatenation of sets of prudent path. Therefore, if $(\mathcal{A}_i)_{i=1}^r$  are $r$ sets of paths such that any sequence in $\mathcal{A}_1\times \dots\times \mathcal{A}_r$ is concatenable, then  $\mathcal{A}_1\oplus \dots \oplus \mathcal{A}_r$ contains all paths obtained by concatenating sequences in   $\mathcal{A}_1\times \dots\times \mathcal{A}_r$.

\subsection{Partially directed self-avoiding walk (PDSAW)}
\label{sec:IPDSAW}
The  \emph{partially directed self-avoiding walk} is a random walk on $\mathbb Z^2$ whose increments are unitary and can take only three possible directions. For instance, when the increments of the path  are chosen in $\{\uparrow,\, \downarrow,\, \rightarrow\}$, then the path is  \emph{west-east} oriented.  By rotating an west-east path by $\pi/2$ radians we obtain a \emph{south-north} path, whose increments are 
chosen in  $\{\uparrow,\, \leftarrow,\, \rightarrow\}$, see Figure \ref{fig:IPDRW1bis} for two examples of such paths.
  By repeating  twice this rotation, we recover the  \emph{east-west} and the \emph{north-south} paths.  
  In what follows and for $L\in \N$,  the set of 
{west-east} partially directed  paths of length $L$ ({south-north}, {east-west},  {north-south} respectively)  
 will be denoted by $\Omega_{L,pd}^{\rightarrow}$ ($\Omega_{L,pd}^{\uparrow}$, $\Omega_{L,pd}^{\leftarrow}$, $\Omega_{L,pd}^{\downarrow}$ respectively).

\begin{definition}[Inter-stretch]
\label{def:interStretches}
We call  \emph{inter-stretch} every increment in the direction which gives the orientation of a given partially directed path. Therefore,
any partially directed path of finite length can be partitioned into $(N-1)$-\emph{inter-stretches} and $N$-\emph{stretches}, $(\ell_1,\dots,\ell_N)\in \mathbb Z^N$, for some $N\in\mathbb N$.
For $i\in \{1,\dots,N\}$, the modulus of $\ell_i$ gives the number of unitary steps composing the $i$-th stretch  and when $\ell_i\neq 0$,  the sign of $\ell_i$ gives its orientation.
In a west-east or east-west  path, we say that $\ell_i$ has a south-north orientation ($\uparrow$) if $\ell_i>0$ and north-south ($\downarrow$) if  $\ell_i<0$. 
In a south-north or north-south path, we say that $\ell_i$ has an west-east orientation ($\rightarrow$) if $\ell_i>0$ and east-west ($\leftarrow$) if  $\ell_i<0$ (see Figure \ref{fig:IPDRW1bis}).  Thus, e.g., 
$$
\Omega_{L,pd}^{\to}=
\bigcup_{N=1}^L\Big\{\ell=(\ell_i)_{i=1}^{N}\in \mathbb{Z}^{N} \colon\, N-1+|\ell_1|+\dots+|\ell_N|=L\Big\}.
$$
\end{definition}

\begin{remark}
In this paper we also take into account  those partially directed path with only one vertical stretch and zero inter-stretches (thus $N=1$ in Definition \ref{def:interStretches}). This is a slight difference with respect to \cite{CGP13}), in which $N\geq 2$.
\end{remark}

In Section \ref{sec:NE-IPSAW} we define the {two-sided path}. They are obtained by concatenating alternatively, e.g.,
some {west-east}  partially directed paths with  some {south-north} partially directed paths. 
 However, concatenating such oriented 
path requires an additional geometric constraint called \emph{exit-condition} which requires a proper definition.

\begin{definition}[Exit condition]
\label{def:exit}
Let $N\in \N$ and let $\ell=(\ell_1,\dots,\ell_N)\in \mathbb{Z}^N$ be an arbitrary 
sequence of stretches.  Then, $\ell$ satisfies the \emph{upper exit condition}
if its  last 
stretch finishes strictly above all other stretches, i.e.,
$$
 \ell_1+\cdots+\ell_N\, >\, \max_{0\leq i< N}\{\ell_1+\cdots+\ell_i\},$$
and $\ell$  satisfies the \emph{lower exit condition}
or if its last 
stretch finishes strictly below all other stretches, i.e.,
$$
 \ell_1+\cdots+\ell_N\, <\, \min_{0\leq i< N}\{\ell_1+\cdots+\ell_i\}.$$
 \end{definition}

\begin{definition}[Oriented blocks]
\label{def:OrientedBlock}
An arbitrary \emph{west-east} partially directed path  $(\ell_1,\ell_2,\dots,\ell_N)$  is  
called  \emph{upper oriented}  if 
its first stretch is negative and if it obeys the \emph{upper exit condition} (see Figure \ref{fig:IPDRW1bis} (A)). 
Otherwise, it is called  \emph{lower oriented} if its first stretch is positive and if it obeys the \emph{lower exit condition}.
We denote by $\mathcal{O}_L^{\to,+}$  the set of \emph{upper west-east} oriented blocks of size $L$   and by  and by $\mathcal{O}_L^{\to,-}$   the set of 
\emph{lower west-east} oriented blocks, i.e.,  
 \begin{align}\label{defoo}
 \mathcal{O}_L^{\to,+}&:=\{\ell \in \Omega_{L,pd}^{\to}\colon \ell_1<0\ \text{and $\ell$ satisfies the \emph{upper exit condition}}\},\\
  \mathcal{O}_L^{\to,-}&:=\{\ell \in \Omega_{L,pd}^{\to}\colon \ell_1>0\ \text{and $\ell$ satisfies the \emph{lower exit condition}}\}.
 \end{align}
We define analogously  the sets $\mathcal{O}_L^{\uparrow,+}$ and $\mathcal{O}_L^{\uparrow,-}$ of \emph{upper south-north} oriented blocks and of \emph{lower south north} oriented blocks, respectively, and so on.

We stress that for satisfying the exit condition it must hold that $N\geq 2$, i.e., we need at least two stretches. 

 \end{definition}

\subsection{Two-sided prudent path}
\label{sec:NE-IPSAW}

With the oriented blocks (recall definition \ref{def:OrientedBlock}) in hand, we can 
define a larger class of prudent paths: the {2-sided} prudent paths, which ultimately will constitute the  building bricks of the prudent path.  
Those $2$-sided prudent path have a general orientation that 
can be \emph{north-east} (NE), \emph{north-west} (NW), \emph{south-west} (SW) or \emph{south-east} (SE).
In the rest of the section we focus on NE-prudent path, but all definitions we give can easily be adapted 
to consider a generic oriented (NE, NW, SE, SW) prudent self-avoiding path.

As mentioned above, {north-east} prudent path  are obtained by considering a family of 
{west-east} oriented blocks and a family of 
{south-north} oriented blocks and by concatenating them alternatively.  

%
%
%
%


\smallskip

\label{Sec:OriBlocks}

  \begin{figure*}
  \begin{subfigure}[t]{0.35\textwidth}
\includegraphics[scale=1.2]{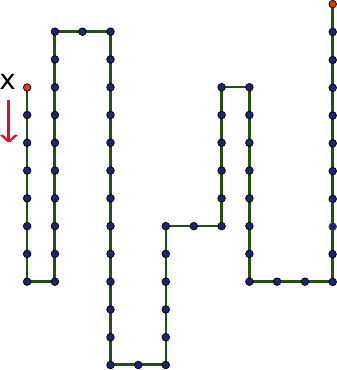}
\caption{A west-east block.}
\end{subfigure}
  \begin{subfigure}[t]{0.35\textwidth}
\includegraphics[scale=1.2]{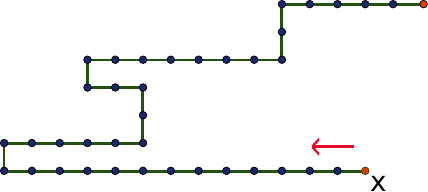}
\caption{A south-north block.}
\end{subfigure}
\caption{
The west-east oriented block ($\textrm A$) is made of  $12$ stretches 
and  is upper oriented since $\ell_1<0$ and $\ell_1+\dots+\ell_{12}>\displaystyle\max_{1\leq i\leq 11}\{\ell_1+\dots +\ell_i\}$. \\
Analogously, the south-north block ($\textrm B$) is upper oriented as well.}
\label{fig:IPDRW1bis}
\end{figure*} 

%
 
 
 
 
\begin{definition}[NE-prudent path]
\label{def:NEPrudentPath}
To define a {NE-prudent self-avoiding path} of length $L\in \N$ we consider $r\in \N$ oriented blocks, $(\pi_1,\dots, \pi_r)$, of length $t_1,\dots , t_r$ respectively, with $t_1+\cdots + t_r=L$ and $t_i\geq 4$. 
We assume that those blocks indexed by odd integers are either all {upper west-east} oriented (in which case all blocks 
indexed by even integers are {upper south-north} oriented) or all {upper south-north} oriented (in which case all blocks 
indexed by even integers are {upper west-east} oriented). 
In definition \ref{def:OrientedBlock} we have imposed that an {upper} oriented block starts with a negative stretch but this constraint can be relaxed for $\pi_1$ (the 
first oriented block of the sequence). We have also imposed that an {upper} oriented block satisfies the {upper  exit condition}  but  this constrain can be relaxed for $\pi_r$ (the last block of the sequence).  See Figure \ref{fig:IPDRW2} for an example of a NE-prudent path 
with these $2$ constraints relaxed. 
Then, we concatenate $\pi_1,\dots,\pi_r$ (which is possible because the first $r-1$ blocks satisfy the exit condition) 
and the resulting path  is denoted by $\pi_1\oplus\dots\oplus \pi_r$. We call such path a \emph{NE-prudent self-avoiding path}, see Figure \ref{fig:IPDRW2}. The sequence $(\pi_1,\dots, \pi_r)$ is called the \emph{block decomposition} of the path and it is unique.
\end{definition}

We now provide a formal definition of $\Omega_L^{\tx{NE}}$:
\begin{align}\label{defNE}
\Omega_L^{\tx{NE}}=& \bigcup_{ r\in 2\N} \bigcup_{t_1+\dots+t_r=L} \big[ \mathcal{O}_{t_1,*}^{\to,+} \oplus  \mathcal{O}_{t_2}^{\uparrow,+}
\oplus \dots\oplus \mathcal{O}_{t_{r-1}}^{\to,+}\oplus \mathcal{O}_{t_r, \square}^{\uparrow,+}\Big]
\cup \Big[\mathcal{O}_{t_1,*}^{\uparrow,+} \oplus  \mathcal{O}_{t_2}^{\to,+}
\oplus \dots\oplus \mathcal{O}_{t_{r-1}}^{\uparrow,+} \oplus \mathcal{O}_{t_r,\square}^{\to,+} \Big]\\
\nonumber & \cup \ \bigcup_{ r\in 2\N-1} \bigcup_{t_1+\dots+t_r=L} \big[ \mathcal{O}_{t_1,*}^{\to,+} \oplus  \mathcal{O}_{t_2}^{\uparrow,+}
\oplus \dots\oplus \mathcal{O}_{t_{r-1}}^{\uparrow,+} \oplus \mathcal{O}_{t_r,\square}^{\to,+}\Big]
\cup  \Big[\mathcal{O}_{t_1,*}^{\uparrow,+} \oplus  \mathcal{O}_{t_2}^{\to,+}
\oplus \dots\oplus \mathcal{O}_{t_{r-1}}^{\to,+} \oplus \mathcal{O}_{t_r,\square}^{\uparrow,+} \Big],
\end{align}
where the notations $\mathcal{O}_{t,*}^{\, \cdot\, ,+}$  means that the condition $\ell_1<0$ has been removed from \eqref{defoo}
 and $\mathcal{O}_{t,\square}^{\, \cdot\, ,+}$ means that the exit condition has been removed from \eqref{defoo}.

\begin{figure}
\includegraphics[scale=1.3]{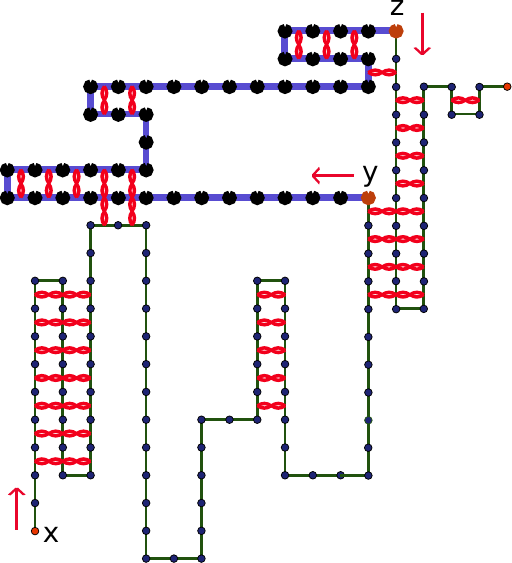}
\caption{A NE-PSAW path made of  three blocks: 
the first and the third blocks are west-east (in green) and the second block is south-north (in blue). 
The first block starts at $x$, the second block starts at $y$ and the third block starts at $z$. Their orientation 
is given by the arrow.
Interactions in each block and between different blocks are highlighted in red. }
\label{fig:IPDRW2}
\end{figure}

\begin{remark}
\label{rem:IPDSAW-IPSAW}
Let us observe that indeed $\Omega_{L,pd}^\to$ and $\Omega_{L,pd}^\uparrow$  are  NE-prudent self-avoiding walk. It corresponds to the case in which we have only one block, i.e., $r=1$.  
\end{remark}

\smallskip

\begin{figure}
\includegraphics[scale=1.2]{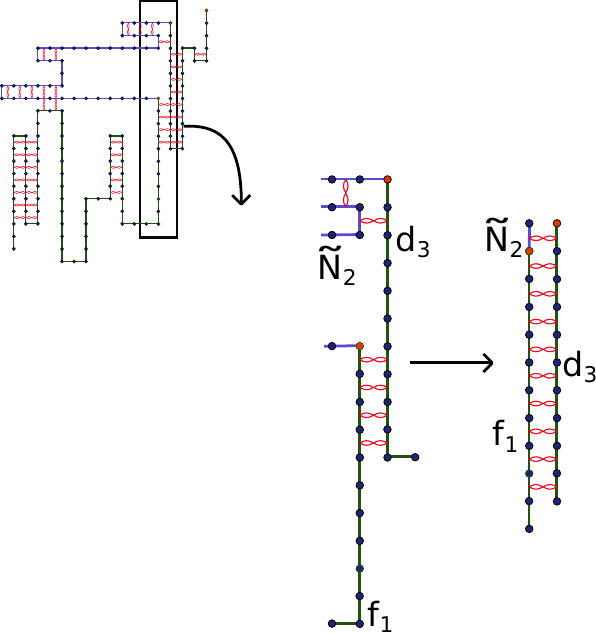}
\caption{On the left, a NE-PSAW path made of three blocks. 
In the picture we zoom in on the interactions between the third block and the 
rest of path. 
We recall that the third block can only interact with its two preceding blocks, i.e., the first and the second one.
We call $f_1$ the last vertical stretch of the first block and
$d_3$ the first vertical stretch of the third block.
 The interactions between the first and the third blocks involve   $f_1$ and $d_3$ while the interactions 
 between the second and the third blocks involve $d_3$ and  $\tilde N_2$ (the number of inter-stretches in the second block that 
 may truly interact with $d_3$, on the picture $\tilde N_2=1$).
Such interactions are bounded above by $(\tilde N_2 + f_{1}) \vl d_{3}$. }
\label{fig:IPDRW3}
\end{figure}

\subsection{Interacting prudent self-avoiding walk}
\label{sec:IPSAWdef}
In this section we show how a general prudent path can be decomposed in a unique manner into a 
sequence of {2-sided} prudent paths called {macro-blocks}. There is a difference between the decomposition of a two-sided path into 
oriented blocks and that of a generic prudent path into macro-blocks. We have indeed seen in Section \ref{def:NEPrudentPath} above that 
the exit condition, which is an intrinsic constraint, was sufficient to make sure that 
oriented blocks alternatively {west-east} and {south-north} are concatenable. However, to make sure that a given family of 
2-sided prudent paths is concatenable, one can not rely on some intrinsic geometric constraint anymore. Such a family must indeed satisfy a global constraint, that is, 
each 2-sided prudent path has to satisfy the prudent condition with the all path it will be attached to
and this condition is not intrinsic anymore,  see Figure \ref{fig:IPDRW9}. 

%
 We recall that a walk is said to be \emph{prudent} if none of its steps point in the direction of its range.
 In the sequel we refer to this constraint as the \emph{prudent condition}.

\subsubsection{Macro-block decomposition}
\label{sec:macrobloks}
Let us start by noticing that a prudent walk can be viewed as a sequence of 
NE, NW, SE, SW two-sided sub-paths that we will call  \emph{macro-block}, see Figure \ref{fig:IPDRW9}.   

 \begin{definition}\label{def:ThetaLL}
For very $m,L\in \mathbb N$ we denote by $\Theta_{m,L}$ the set gathering all concatenable sequences of $m$ {two-sided} paths such that the cumulated length of the {two-sided} paths in the sequence is $L$ and such that:
\begin{enumerate}
\item  two consecutive {two-sided} paths in the sequence do not have the same orientation,
\item  the first $m-1$ {two-sided} paths in the sequence contain at least $2$ oriented blocks.
\end{enumerate}
For the ease of notation, we recall \eqref{defNE} and we let 
$\Omega_{L,\triangle}^{\tx{NE}}$ be the set of \emph{north-east} prudent path containing at least two oriented blocks (the same definition holds with the $3$ others possible orientations of a \emph{two-sided} path). 
Thus, 
\begin{align}\label{THETAml}
\Theta_{m,L}=\hspace{-.2cm}\bigcup_{t_1+\dots+t_m=L}\  \bigcup_{\substack{(x_i)_{i=1}^m\in \{\tx{NE}, \tx{NW}, \tx{SE}, \tx{SW}\}\\
x_{i-1}\neq x_{i},\ i\leq r}} 
\begin{aligned}
\Big\{(\Lambda_1,\dots,\Lambda_m)\in\, \Omega_{t_1,\triangle}^{x_1}\times \dots & \times\Omega_{t_{m-1},\triangle}^{x_{m-1}}\times \Omega_{t_m}^{x_m}\colon \\
 &(\Lambda_1,\dots,\Lambda_m)\   \text{is concatenable}\, \Big\}
 \end{aligned}
\end{align}
\end{definition}
Finally, we observe that any prudent path of length $L$ can be decomposed into a sequence of macro-blocks in  $\cup_{m\geq 1} \Theta_{m,L}$   and moreover, thanks to 
the conditions (1) and (2) in Definition \ref{def:ThetaLL} we can assert that such decomposition is unique. 
Therefore, we may partition $\Omega_L^{\tx{PSAW}}$ as 
\begin{align}\label{eq:OmegaPsaw}
\Omega_L^{\tx{PSAW}}=\cup_{m\geq 1} \{\Lambda_1\oplus\dots\oplus \Lambda_m \colon (\Lambda_1,\dots,\Lambda_m)\in \Theta_{m,L}\}
\end{align}
An example of such decomposition is provided in Figure \ref{fig:IPDRW9}. 
 


\subsubsection{Upper bound on the number of \emph{macro-block} in the decomposition of a generic prudent path.}
\label{sec:macrobloks}
The prudent condition imposes strong constraints on the number of macro-block composing the path:
if we consider the smallest rectangle embedding the whole path, then whenever the random walk wants to start a new
 macro-block, it must cross the whole rectangle in one direction and in such direction the length of the rectangle is increased  by at least one unit.
Therefore the longer it is the path, the harder (expensive) it becomes to start a new macro-block. In Lemma \ref{lemma:macroblocknumber} we provide an upper bound on the number of macro-blocks 
in a prudent path of a given length.

\begin{figure}
\includegraphics[scale=0.75]{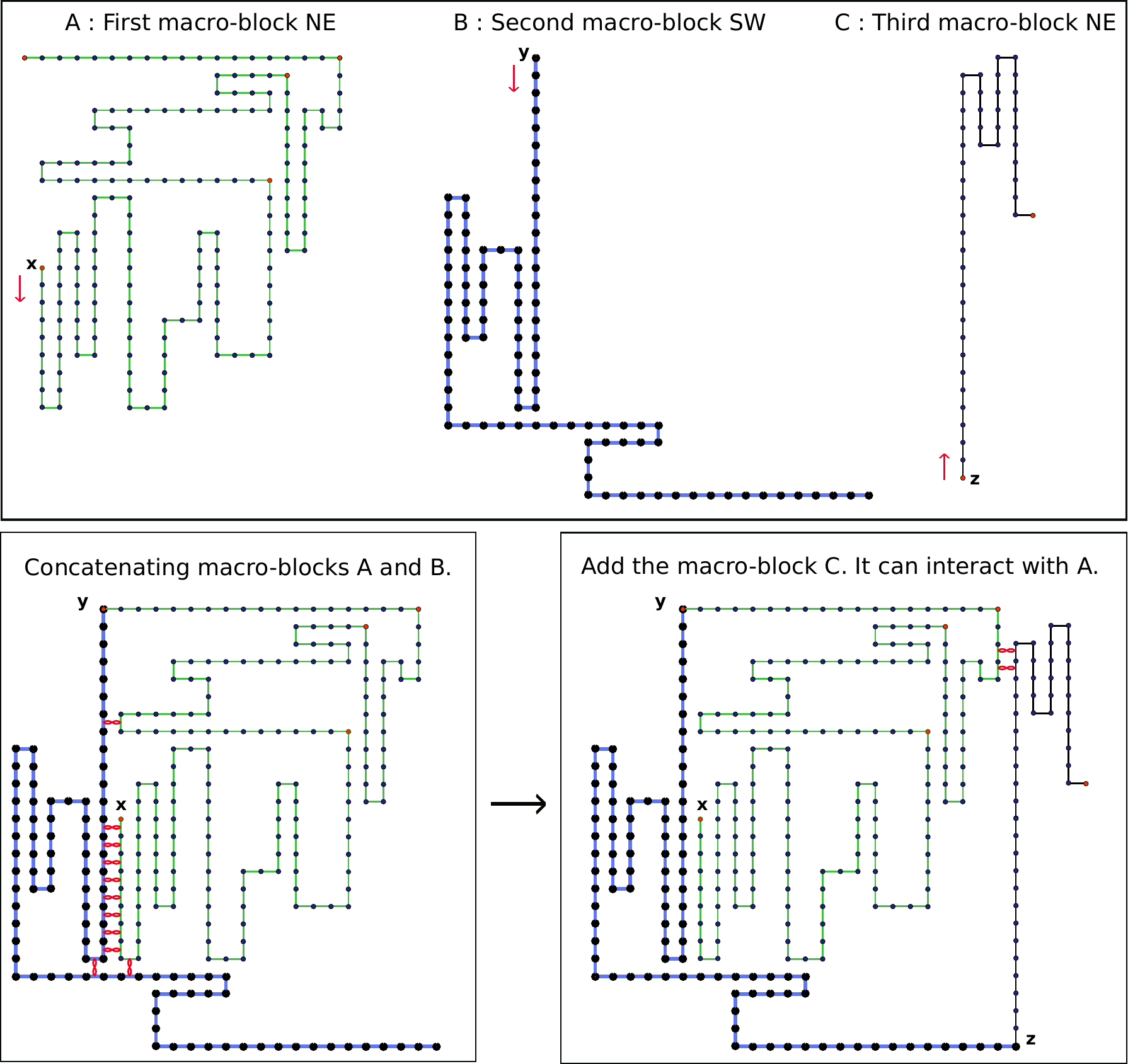}
\caption{Decomposition of a prudent walk into macro-blocks.  
In the picture we have a sequence of three macro-blocks, $A, B,\, \text{and}\, C$. The first macro-block (A) has a NE-orientation. The second block (B) has a SW-orientation and it is \emph{compatible} with the first macro-block, that is, the \emph{prudent condition} is satisfied. 
This allows us to  concatenate A with B.
The third macro-block (C) has a NE-orientation and it satisfies the compatibly condition with  $A\oplus B$.
The interaction between macro-blocks are highlighted in red.
}
\label{fig:IPDRW9}
\end{figure}

\begin{lemma}
\label{lemma:macroblocknumber}
Let $L$ be the path length. Then the number of macro-blocks composing the path is bounded from above by $\mathcal{O}(\sqrt{L})$.
\end{lemma}
\begin{proof}
Pick $w\in \Omega_L$, and let $r$ be the number of macro-blocks in $w$. For $j \in \{1\dots,r\}$, we denote by
$R_j$ the smallest rectangle containing the first  $j$ macro-blocks of $w$.  In order to complete the $j$-th macro-block and to start a new one, the path should either 
cross $R_j$ horizontally and increase the width of $R_j$ by at least $1$ or vertically and increase the height of 
$R_j$ by at least $1$. Therefore, we define $n_v$ the number of times that a macro-block ends
with a vertical cross, and $n_h$ its horizontal counterpart. As a consequence, by keeping in mind that $w$ has length $L$, it must hold that
\begin{equation}\label{length}
\sum_{i=1}^{n_v} i+ \sum_{j=1}^{n_h} j\leq L.
\end{equation} 
From \eqref{length} it comes that $n_v (n_v+1)+n_h (n_h+1)\leq 2L$ and therefore
$n_v^2+n_h^2\leq 2L$. Under such condition, the quantity $n_v+n_h$ is maximal
when $n_v=n_h=\sqrt{L}$. Thus, the number of macro-blocks made by $w$ is not larger than 
$2 \sqrt{L}$.

\end{proof}


\section{Proof of Theorem \ref{Thm1}}
\label{proof of Thm1}

In this section we prove Theorem \ref{Thm1} subject to Theorem \ref{thh1} which ensures that $\mathrm F^{\tx{NE}}(\beta)=\mathrm F(\beta)$ for any $\beta\in [0,\infty)$. Therefore it is sufficient  to prove Theorem \ref{Thm1} for NE-PSAW.  
Theorem \ref{thh1} will be proven in Section \ref{proof of thh1}.

\smallskip

We consider the free energy of NE-IPSAW
\begin{equation}
\mathrm F^{\tx{NE}}(\beta):=\lim_{L\to\infty} \frac{1}{L}\log \mathrm Z_{\beta,\, L}^{\tx{NE}}.
\end{equation}
In Section \ref{app:freeenergy} we prove that this limit exists and is finite.
Let us observe that, by Remark \ref{rem:IPDSAW-IPSAW}, $\mathrm F^{\tx{NE}}(\beta)\geq \mathrm F^{\tx{IPDSAW}}(\beta)$, thus it follows that  $\mathrm F^{\tx{NE}}(\beta)\geq \beta$ cf. (1.9) in \cite{CGP13}. 
To complete the proof of Theorem \ref{Thm1} we have to show that there exists a $\beta_0$ such that 
$\mathrm Z_{\beta,\, L}^{\tx{NE}}\leq e^{\beta (L+o(L))}$ for any $\beta \geq \beta_0$ and $L\in \mathbb N$.
To that purpose we disintegrate   the partition function $\mathrm Z_{\beta,\, L}^{\tx{NE}}$ by 
using the decomposition of any $L$-step NE-PSAW path $\pi$  into a family of oriented blocks $(\pi_1,\dots, \pi_r)$
with $r\leq L/4$ (cf. Definition \ref{defNE}).  As displayed in \eqref{defNE}, we can distinguish between 4 types of  NE-PSAW paths
depending on the orientation of their first and last oriented block. For
simplicity we will only consider  $\widehat Z_{\beta,\, L}^{\, \, \tx{NE}}$ which is computed by restricting the partition function to those paths starting with a west-east block and ending with a south-north block 
(this corresponds to the first decomposition in \eqref{defNE}). The contribution to $Z_{\beta,\, L}^{\tx{NE}}$ of those path satisfying one of the  
$3$ other possible decompositions in \eqref{defNE} are handled similarly. Therefore,
\begin{align}\label{defNEE}
\widehat Z_{\beta,L}^{\, \, \tx{NE}}&=\sum_{ r\in 2\N} \sum_{t_1+\dots+t_r=L} \sum_{(\pi_1,\dots,\pi_r)\in \mathcal{O}_{t_1,*}^{\to,+} \times \,  \mathcal{O}_{t_2}^{\uparrow,+}
\times \dots\times\,   \mathcal{O}_{t_{r-1}}^{\to,+}\times \,  \mathcal{O}_{t_r, \square}^{\uparrow,+}} \exp\bigg\{\, \beta\sum_{j=1}^r H(\pi_i)+\beta \ \Phi(\pi_1,\dots,\pi_r)\bigg\},
\end{align}
where $\Phi(\pi_1,\dots,\pi_r)$ is a suitable function that takes into account the interactions between different oriented blocks, i.e.,
counts the number of self-touchings involving monomers belonging to two different oriented blocks.

Henceforth, for every $i\in \{1,\dots,r\}$ we let $d_i$ (respectively $f_i$) be the first stretch (resp. last stretch) of $\pi_i$ 
and we let $N_{i}$ be the number of stretches constituting $\pi_i$.
We note that  $\phi(\pi_1,\dots,\pi_r)$ can be computed by summing for $i=1,\dots, r-1$  the number of self-touchings between 
$\pi_{i+1}$ and the sub-path $\pi_1 \oplus\dots \oplus \pi_{i}$. Moreover,
the prudent condition implies that $\pi_{i+1}$ can interact with  
 $\pi_1 \oplus\dots \oplus \pi_{i}$ only through $\pi_{i-1}$ and $\pi_i$. 
To be more specific (see Figure  \ref{fig:IPDRW3}), the self-touchings between $\pi_i$ and $\pi_{i+1}$ may only happen between $d_{i+1}$ (the first stretch of $\pi_{i+1}$)  and some of the inter-stretches of $\pi_i$ (whose number is denoted by $\tilde N_i$), while the self-touchings between $\pi_{i-1}$ and  $\pi_{i+1}$   may only happen between $d_{i+1}$ and $f_{i-1}$ (the last stretch of $\pi_{i-1}$).
Of course, for every $i\in \{0,\dots,r-1\}$, the number of inter-stretches in $\pi_i$ that may interact with $d_{i+1}$  is not larger than  the number of inter-stretches in $\pi_i$,  i.e., $\tilde N_{i}\leq N_{i}-1$.
By assigning to $\tilde N_i$ the same sign as $f_{i-1}$, we can check without further difficulty (see Figure \ref{fig:IPDRW3}) that the number of 
self-touchings between $\pi_{i-1},\pi_i$ and $\pi_{i+1}$ is bounded from above by 
 \begin{equation*}
 (\tilde N_{i} + f_{i-1}) \vl d_{i+1},
 \end{equation*}
 where the $\vl$ operator is defined in \eqref{eq:defVL} below.
 We stress again that $\tilde N_{i}$ and  $f_{i-1}$ have the same sign, while $d_{i+1}$ has the opposite orientation. 
 By using the definition of $\vl$ in \eqref{eq:defVL} and the triangle inequality, we have the following inequality  for every $c\in (0,1/2)$, i.e.,    \begin{equation}\label{eq2}
(\tilde N_{i} + f_{i-1}) \vl d_{i+1}\, \leq \,
 \frac{1}{2}|d_{i+1}| + \frac{1}{2}|f_{i-1}| + \left(\frac{1}{2}+c\right)| N_{i}-1 | - c\, |f_{i-1} + d_{i+1}|, 
\qquad i=1,\cdots,r-1,
\end{equation}
where $f_0=0$ by definition. It turns out that the value of $c$ is worthless: in the sequel we choose $c=1/4$.
We use \eqref{eq2} to conclude that
\begin{equation}\label{eq:InterEst1}
e^{\,\beta\,  \Phi(\pi_1,\dots,\pi_r)}\leq e^{\, \frac{\beta}{2}\left(|d_2|+\cdots+|d_r|\right)}e^{\, \frac{\beta}{2}\left(|f_1|+\cdots+|f_{r-2}|\right)}
e^{\, \frac{3}{4}\,\beta\, \left(N_1+\cdots+N_{r}-r\right)}
e^{\, -\frac{\beta}{4}\left(|f_0+d_2|+\cdots + |f_{r-2}+d_r|\right)}.
\end{equation}

At this stage, we let $\mathrm{Q}_{\beta, t,d,f,N}$ be the partition function associated with those oriented blocks made of 
$N$ stretches $(\ell_1,\dots,\ell_N)$, of total length $t$,   
starting with a stretch $\ell_1=d$, finishing with a stretch $\ell_N=f$. 
Since $\mathrm{Q}_{\beta, t,d,f,N}$ is a partition function involving partially directed paths only, we can use  
the Hamiltonian representation displayed in \cite{CGP13} with the help of the operator 
 $\vl$: for any pair $(x, y)\in \mathbb Z^2$ we let 
\begin{equation}\label{eq:defVL}
x \vl y:= \frac{1}{2}\big(\,|x|+|y|-|x+y|\, \big)
=\begin{cases}
\min\{|x|,|y|\}, &\, \text{if}\,\, xy <0,\\
0, &\, \text{otherwise}.
\end{cases}
\end{equation} 
 In such a way for a given sequence of $N$-stretches, $(\ell_1,\dots,\ell_N)$, the Hamiltonian in \eqref{eq:ham1} becomes
\begin{equation}
\label{eq:ham12}
\textrm H\big((\ell_1,\dots,\ell_N)\big)={\,\sum_{i=1}^{N-1}\, (\ell_{i} \vl \ell_{i+1})}.
\end{equation}
Since we are looking for an upper bound on  $\widehat  Z_{\beta,\, L}^{\,\,\tx{NE}}$, we forget about the exit condition
 that a block must satisfy (cf. Definition \ref{def:exit}) and we define $\mathrm{Q}_{\beta,\, t,d,f,N} $, on 
 $\mathcal{L}_{N,t}$, the set of all partially-directed paths of length $t$ with $N-1$
  inter-stretches. To be more specific, for $N\in\mathbb N$ we let
\begin{equation}
\label{eq:Lnl}
\begin{split}
&\mathcal{L}_{N,t}\, :=
 \left\{\ell=(\ell_1,\dots, \ell_N)\, :\, \sum_{i=1}^N\, |\ell_i|=t-N+1 \right\},
\end{split}
\end{equation}
and we define
\begin{equation}
\label{eq:deQ}
\mathrm{Q}_{\beta, t, d,f,N} \, := \, 
 \sum_{\substack{\ell \in \mathcal{L}_{N,t} \\ \ell_1=d, \ell_N=f}}  \exp\left\{\, \beta\, \sum_{n=1}^{N-1}\, (\ell_n \vl \ell_{n+1})\right\}.
\end{equation}
It follows that an upper bound on $ \widehat Z_{\beta,\, L}^{\,\,\tx{NE}}$ can be obtained from \eqref{defNEE}. To that aim,  for a given $r\in\{1,\dots, L/4\}$ 
and  $t_1+\dots+t_r=L$, we rewrite the 
inner summation in \eqref{defNEE} 
depending on the value taken by $(d_i,f_i,N_i)$ for $i\in \{1,\dots,r\}$.  We recall that $d_i<0$ for $i\geq 2$ and we lighten the notation with
$$
\Xi_{(t_1,\dots,t_r)}=\Big\{
(d_i,f_i, N_i)_{i=1}^r \colon |d_i|+|f_i|+N_i-1\leq t_i, \,\, d_i<0\, \forall\, i\geq 2,\,\, N_i \geq 2\, \forall\, i\neq r\,
\Big\},$$
where the $(t_1,\dots,t_r)$-dependency of  $\Xi$ may be omitted when there is no risk of confusion.
We  plug \eqref{eq:InterEst1} inside \eqref{defNEE} to obtain
\begin{equation}\label{eq:Zest1}
\begin{split}
\widehat Z_{\beta,\, L}^{\,\, \tx{NE}}\,& \leq \,
\sum_{r=1}^{L/4} \sum_{{t_1+\cdots+t_r=L}} \sum_{(d_i,f_i,N_i)_{i=1}^r\in \Xi}
\left(\prod_{i =1}^r\, \mathrm{Q}_{\beta,\, t_i,d_i,f_i,N_i}\right) \, \times
\\
& \qquad
e^{\, \frac{\beta}{2}\left(|d_2|+\cdots+|d_r|\right)}e^{\, \frac{\beta}{2}\left(|f_1|+\cdots+|f_{r-2}|\right)}
e^{\, \frac{3}{4}\,\beta\, \left(N_1+\cdots+N_{r}-r\right)}
e^{\, -\frac{\beta}{4}\left(|f_0+d_2|+\cdots + |f_{r-2}+d_r|\right)}
\end{split}
\end{equation}


\begin{remark}
According to Definition \ref{def:OrientedBlock} and \ref{def:exit}, we want to stress that 
$\pi_r$, 
the last block of the path, can have zero inter-stretches, i.e., it may happen that 
$N_r=1$. For the other blocks, $\pi_1,\dots,\pi_{r-1}$, $N_i$ must be larger or equal to $2$, because the exit condition (cf. Definition \ref{def:exit}) implies that each such block  contains at least two stretches.
\end{remark}
%

With the help of \eqref{eq:defVL}  we can rewrite $\mathrm{Q}_{\beta,\, t,d,f,N}$ in \eqref{eq:deQ}   as 
\begin{equation}
\label{eqdef:QfdtN}
\begin{split}
\mathrm{Q}_{\beta,\, t,d,f,N} \, &
 = \, \sum_{\substack{\ell \in \mathcal{L}_{N,t}\\ \ell_1=d, \ell_N=f}} 
\exp\left\{\, \beta\, \sum_{n=1}^{N}\, |\ell_n|\, - \frac{\beta}{2}\, 
\sum_{n=1}^{N-1}\, |\ell_n+\ell_{n+1}|\, -\frac{\beta}{2}|f| -\frac{\beta}{2}|d| \right\}.
\end{split}
\end{equation}
Recall \eqref{eq:Lnl}. For every $\ell\in \mathcal L_{N,t}$, the equality  $\sum_{n=1}^{N}\, |\ell_n|= t-N+1$ can be plugged into 
\eqref{eqdef:QfdtN} to obtain
\begin{equation} 
\label{eqdef:QfdtLastLine}
\mathrm{Q}_{\beta,\, t,d,f,N} \, 
= \, e^{\, \beta\, \left(\, t-N+1 -\frac{1}{2}|f| -\frac{1}{2}|d| \, \right)} \, 
\sum_{\substack{\ell \in \mathcal{L}_{N,t}\\ \ell_1=d, \ell_N=f}} 
\prod_{n=1}^{N-1}\exp\left\{ - \frac{\beta}{2} \,  |\ell_n+\ell_{n+1}| \right\}.
\end{equation}  
According to the method used in \cite[Section 2.1]{CGP13}, the r.h.s. of \eqref{eqdef:QfdtLastLine} admits a probabilistic representation.  
Let us introduce a random walk $V:=(V_i)_{i\in \N}$ with i.i.d. increments  $(U_i)_{i\in\mathbb N}$ following a discrete Laplace distribution, i.e., 
\begin{equation}
\label{eq:def_V}
\mathrm P_\beta \left( U_1=k\right) \,=\, \frac{e^{-\frac{\beta}{2} \, |k|}}{c_\beta}, \quad k\in \mathbb Z,
\end{equation}
where $c_\beta$ is the normalization constant, i.e.,
\begin{equation}
\label{eq:def_cb}
c_{\beta}\,=\,\sum_{k\in\mathbb Z} e^{-\frac{\beta}{2} \, |k|}\,=\,\frac{1+e^{-\frac{\beta}{2}}}{1-e^{-\frac{\beta}{2}}}.
\end{equation}
In such a way the relation $V_i=(-1)^{i-1} \ell_i$ for $i=0,\dots,N$ which is equivalent to
\begin{equation}
\label{eq:def_U}
U_i=(-1)^{i-1}(\ell_{i-1}+\ell_{i}), \quad \text{for} \quad i=1,\cdots, N,
\end{equation}
 with $\ell_{0}=0$, defines a one-to-one map between $\mathcal{L}_{N,t}$ and the set of all possible random walk paths of 
length $t$ and geometric area $G_N(V)$ that satisfies
\begin{equation}
\label{eq:GeomArea}
G_N(V):=\sum_{n=1}^{N} |V_n|=t-N+1.
\end{equation}
Therefore  \eqref{eqdef:QfdtLastLine} becomes
\begin{equation}\label{eq:Q}
\mathrm{Q}_{\beta,\, t,d,f,N} =
c_{\beta}^{N-1}
\,  e^{\, \beta\, \left(\, t-N+1 -\frac{1}{2}|f| -\frac{1}{2}|d| \, \right)} \, 
{ \mathrm P_{ \beta}
\left( \left. G_N(V)=t-N+1,\, V_N=(-1)^{N-1} f\, \right|\, V_1=d\right).
}
\end{equation}
We plug \eqref{eq:Q} into \eqref{eq:Zest1} and
we observe that all the factors $e^{\frac{\beta}{2}|d_i|},\, i=2,\dots,r$ and $e^{\frac{\beta}{2}|f_i|},\, i=1,\dots, r-2$ in the second line of 
\eqref{eq:Zest1}, are simplified by the corresponding quantities appearing in the exponential factor of 
\eqref{eq:Q}, with $f=f_i,\,d=d_i$ and $N=N_i$. 
Since $t_1+\cdots+t_r=L$,
we obtain that
\begin{equation}\label{eq:Zest}
\begin{split}
&
\widehat Z_{\beta,\, L}^{\,\,\tx{NE}}\,  
\leq \,
e^{\,\beta L}\, \sum_{r=1}^{L/4}\, 
\sum_{t_1+\cdots+t_r=L}\, \sum_{(d_i,f_i,N_i)_{i=1}^r\in \Xi}\, 
{ e^{-\frac{\beta}{4}\, (N_1+\cdots+N_{r}-r)}c_{\beta}^{N_1+\cdots+N_r-r} 
}
e^{-\frac{\beta}{2}|d_1|}e^{-\frac{\beta}{2}|f_r|} e^{-\frac{\beta}{2}|f_{r-1}|}\, \times
\\
& \qquad
\prod_{i=1}^r\, 
{
\mathrm P_{ \beta}
 \left( \left. G_{N_i}(V)=t_i-N_i+1,\, V_{N_i}=(-1)^{N_i-1}f_i\, \right|\, V_1=d_i\right)\,
 }
  \prod_{i=0}^{r-2}\, 
 e^{ -\frac{\beta}{4} \left(|f_{i}+d_{i+2}|\right)},
\end{split}
\end{equation}
At this stage we consider the homogeneous Markov chain kernel (recall \eqref{eq:def_V})
\begin{equation}\label{concat}
\rho (x,y):= \frac{e^{\, -\frac{\beta}{4}\left(|x+y|\right)}}{c_{\beta/2}}=
\mathrm P_{\beta/2}\left(\left. \, V_{1}=-y \,\right|\, V_0= x \right),
\end{equation}
where the $\beta$ dependency of $\rho$ is dropped for simplicity.
We observe that $\rho$ is symetric, i.e. $\rho (x,y)=\rho (-x,-y)$.
Since we are working with upper bounds we can safely replace $\beta/2$ in $e^{-\frac{\beta}{2}|f_r|}, e^{-\frac{\beta}{2}|f_{r-1}|}$ and $e^{-\frac{\beta}{2}|d_1|}$ by $\beta/4$
and \eqref{eq:Zest} becomes (with $f_{-1}=f_0=0$ and $d_{r+1}=d_{r+2}=0$)
\begin{equation}
\label{eq:Fin_NEPW0}
\begin{split}
&
\widehat Z_{\beta,\, L}^{\,\, \tx{NE}}\,  
\leq \,
c_{\beta/2}\, e^{\,\beta L}\, \sum_{r=1}^L\, { c_{\beta/2}^r}
 \sum_{t_1+\cdots+t_r=L}\, \sum_{(d_i,f_i,N_i)_{i=1}^r\in \Xi}\,
{ \left(\frac{c_{\beta}}{e^{\frac{\beta}{4}}}\right)^{(N_1+\cdots+N_{r}-r)}
} \times
\\ 
& 
\prod_{i=1}^r\, 
\mathrm P_{ \beta}
 \left( \left. G_{N_i}(V)=t_i-N_i+1,\, V_{N_i}=(-1)^{N_i-1}f_i\, \right|\, V_1=d_i\right)\,
  \, \prod_{i=-1}^{r}\, 
\rho(f_i,d_{i+2}).
\end{split}
\end{equation}
Now, we focus on the second line in  \eqref{eq:Fin_NEPW0}, our aim is to concatenate  all the even blocks on the one hand, and all the odd blocks on the other hand (see Figure \ref{fig:IPDRW10}). 
For this purpose, 
for a given sequence  $(N_1,\dots,N_r)\in \N^r$
and for a given index subset $\mathbf \nu=\{\nu_1,\dots,\nu_m\}\subset \{-1,\dots,r\}$
we set
\begin{equation}
\mathbf N_k:=\sum_{i\in \nu,\, 1\leq i\leq k}\, N_{i}, \qquad  \text{for}\quad k=-1,\dots, r. 
\end{equation}
Note that $\mathbf N_{-1}=\mathbf N_0=0$. We let
$(\tilde{\mathrm P}_{\beta,\nu},\mathbf V)$ be a non-homogeneous random walk  
 $\mathbf V=(\mathbf V_i)_{i=0}^{\mathbf N_r+1}$, starting from $0$,  
for which all increments have law $\mathrm P_{\beta}$ except those  between  $\mathbf V_{\mathbf N_i}$ and $ \mathbf V_{\mathbf N_i+1}$ for 
$i\in\{\nu_1,\dots,\nu_m\} $ that have law $\mathrm P_{\beta/2}$  (cf. \eqref{concat}). In other words ,
\begin{equation}
\begin{split}
&\tilde{\mathrm P}_{\beta,\nu}\Big(\mathbf V_{\mathbf N_i+1}=y\, \mid\, \mathbf V_{\mathbf N_i}=x\Big) =
{\mathrm P}_{\beta/2}\Big(\mathbf V_{1}=y\, \mid\, \mathbf V_{0}=x\Big), \qquad i\in\{\nu_1,\dots,\nu_m\}, 
\\
&\tilde{\mathrm P}_{\beta,\nu}\Big(\mathbf V_{a+1}=y\, \mid\, \mathbf V_{a}=x\Big) =
{\mathrm P}_{\beta}\Big(\mathbf V_{1}=y\, \mid\, \mathbf V_{0}=x\Big), \qquad a\notin  \{\mathbf N_{\nu_1},\dots,\mathbf N_{\nu_m\}}.
\end{split}
\end{equation}
%
We set, for $k\in \{-1,\dots,r\}$,
\begin{equation}
\label{eq:OddEven}
\begin{split}
&\mathbf N_k^e=\sum_{i\in \{1,\dots,k\}\cap 2\N}\, N_{i} 
\qquad \text{and}\qquad 
\mathbf N_k^o=\sum_{i\in \{1,\dots,k\}\cap (2\N-1)}\, N_{i}, 
\end{split}
\end{equation}
We let $(\tilde{\mathrm P}_{\beta}^e, \mathbf V^e),\, (\tilde{\mathrm P}_{\beta}^o, \mathbf V^o)$ be two independent Markov chains of law $\tilde{\mathrm P}_\beta^{e}:=\tilde{\mathrm P}_{\beta,\{-1,\dots,r\}\cap 2\mathbb Z}$ and 
$\tilde{\mathrm P}_\beta^{o}:=\tilde{\mathrm P}_{\beta,\{-1,\dots,r\}\cap (2\mathbb Z+1)}$ respectively. 
We have to look at $(\mathbf V^e_i)_{i=0}^{\mathbf N_r^e+1}$ and $(\mathbf V^o_i)_{i=0}^{\mathbf N_r^o+1}$ as the random walks obtained by
concatenating the even blocks and the odd blocks respectively, see Figure \ref{fig:IPDRW10}.

For a random walk trajectory $V\in \mathbb{Z}^{\N}$ and for two indices $i<j$ we let $G_{i,j}(V):=\sum_{s=i}^j \, |V_s| .$ be the geometric area described by $V$ between $i$ and $j$.
We are now ready to concatenate the even blocks and the odd blocks in  \eqref{eq:Fin_NEPW0}. 
We consider separately the odd and even terms in the second line of \eqref{eq:Fin_NEPW0}. For the odd terms,
since $\rho (x,y)=\rho (-x,-y)$  (cf. \eqref{concat}),
and since for any odd index $i\leq r$, $\mathbf N_i^o=\mathbf  N_{i-2}^o+ N_i$,
the odd terms in the 
integrand of \eqref{eq:Fin_NEPW0} can be rearranged as follows ($f_{-1}=f_0=d_{r+1}=d_{r+2}=0$ by 
definition)
\begin{equation}\label{eq:pasticcio}
\begin{split}
&
\begin{aligned}
\prod_{\substack{i\in \{1,\dots, r\}\\ i\in 2\mathbb Z+1}}\, &
 \mathrm P_{ \beta} \left( \left.  G_{N_i}(V)=t_i-N_i+1,\,  V_{N_i}=(-1)^{N_i-1}f_i\, \right|\, V_1=d_i\right)
  \\
  & \hspace{3cm} \times
  \prod_{\substack{i\in \{-1,\dots, r\}\\ i\in 2\mathbb Z+1}}\, 
\mathrm P_{\beta/2}\Big(V_1=(-1)^{N_i}d_{i+2} \, |\, V_0=(-1)^{N_i-1}f_i \Big) 
\end{aligned}
\\
& 
\qquad =\tilde{\mathrm P}_\beta^{o}\left(
   \begin{aligned} & G_{\mathbf N_{i-2}^o+1,\mathbf N_i^o}(\mathbf V^o)=t_i-N_i+1, \quad
                    \mathbf V_{\mathbf N_{i-2}^o+1}^o=(-1)^{\mathbf N_{i-2}^o}d_i, \quad \mathbf V_{\mathbf N_{i}^o}^o=(-1)^{\mathbf N_i^o-1}f_i,\\
                   & \forall \, i\in\{1,\dots, r\}\cap (2\mathbb Z+1) \quad \text{and} \quad 
                   \mathbf V_{\mathbf N_r^o+1}^o=0
   \end{aligned}
\right),
\end{split}
\end{equation}
An analogous decomposition holds true for the even terms in the 
integrand of \eqref{eq:Fin_NEPW0}.

With the help of \eqref{eq:pasticcio} we interchange the sum over the $t_i$'s with the sum over the $N_i$'s in \eqref{eq:Fin_NEPW0} and we remove the restriction $t_1+\dots+t_r=L$
to obtain the following upper bound, 
\begin{equation}\label{thy}
\begin{split}
&\sum_{t_1+\cdots+t_r=L}\, \sum_{(d_i,f_i,N_i)_{i=1}^r\in \Xi}\,
{ \left(\frac{c_{\beta}}{e^{\frac{\beta}{4}}}\right)^{(N_1+\cdots+N_{r}-r)}
} \times\\
& \quad 
\prod_{i=1}^r\, 
\mathrm P_{ \beta}
 \left( \left. G_{N_i}(V)=t_i-N_i+1,\, V_{N_i}=(-1)^{N_i-1}f_i\, \right|\, V_1=d_i\right)\,
  \, \prod_{i=-1}^{r}\, 
\rho(f_i,d_{i+2})\\
& 
\leq 
\sum_{\substack{N_1+\dots+N_r\leq L+r, \\ N_i\geq 2\, i=1,\dots, r-1}} { \left(\frac{c_{\beta}}{e^{\frac{\beta}{4}}}\right)^{(N_1+\cdots+N_{r}-r)}
} \times  \\
&\quad \sum_{\substack{t_i\colon t_i\geq N_i-1\\ i=1,\dots,r}}
 \tilde{\mathrm P}_\beta^{o}\left(
    G_{\mathbf N_{i-2}^o+1,\mathbf N_i^o}(\mathbf V^o)=t_i-N_i+1,\,
                    \forall \, i\in\{1,\dots, r\}\cap (2\mathbb Z+1)\, \text{and}\, 
                     \mathbf V_{\mathbf N_r^o+1}^o=0
\right) \\
&\hspace{2cm} \times 
\tilde{\mathrm P}_\beta^{e}\left(
   G_{\mathbf N_{i-2}^e+1,\mathbf N_i^e}(\mathbf V^e)=t_i-N_i+1, \,
                    \forall \, i\in\{1,\dots, r\}\cap 2\mathbb Z\, \text{and} \,
                    \mathbf V_{\mathbf N_r^e+1}^e=0
\right) .
\end{split}
\end{equation}
 We note that the sum over the $t_i$'s in the r.h.s. of \eqref{thy} is bounded from above by $1$.
It remains to plug \eqref{thy} into \eqref{eq:Fin_NEPW0} in which we have
exchanged the summation over the $t_i$'s with that over the $N_i$'s. This leads to 
\begin{equation}
\begin{split}
&\widehat  Z_{\beta,\, L}^{\,\, \tx{NE}}\, \leq \, 
c_{\beta/2}
e^{\,\beta L}\, \sum_{r=1}^{L/4}\, {c_{\beta/2}^r}
 \sum_{\substack{N_1+\dots+N_r\leq L+r, \\ N_i\geq 2\, i=1,\dots, r-1}}\,
{ \left(\frac{c_{\beta}}{e^{\frac{\beta}{4}}}\right)^{(N_1+\cdots+N_{r}-r)}
}
\\
\label{eq:estFZnePW2}
&\qquad \leq
c_{\beta/2} e^{\, \beta\, L}\, \left[c_{\beta/2}\sum_{N=0}^{\infty} \left(\frac{c_{\beta}}{e^{\frac{\beta}{4}}}\right)^{N}\right]\sum_{r=0}^\infty  {c_{\beta/2}^r} \left[ \sum_{N=1}^{\infty} \left(\frac{c_{\beta}}{e^{\frac{\beta}{4}}}\right)^{N}\right]^{r}.
\end{split}
\end{equation} 
 At this stage, by using the definition of $c_\beta$ in \eqref{eq:def_cb}, there exists $\beta_0 \in (0,\infty)$ such that $c_\beta/e^{\beta/4}<1/4$ and $c_{\beta/2}\leq 2$, for any $\beta > \beta_0$. 
This implies that  $\widehat  Z_{\beta,\, L}^{\,\, \tx{NE}}\leq C(\beta)\, e^{\, \beta\, L}$ for some suitable constant $C(\beta)\in (0,\infty)$.

\begin{figure}
\includegraphics[scale=0.75]{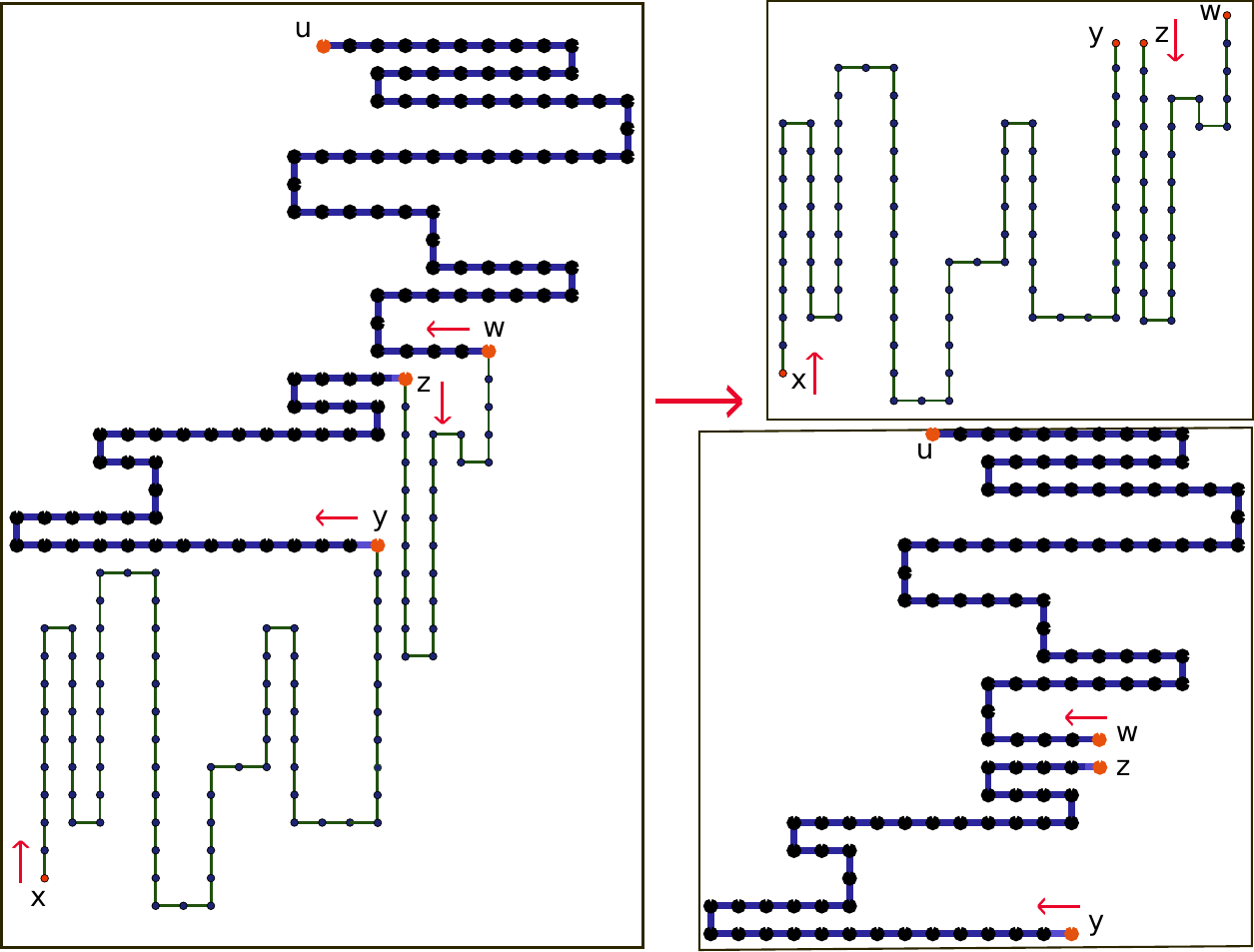}
\caption{A  NE-prudent path made of two south-north blocks (the first and the third, in green) and two west-east blocks (the second and the fourth, in blue).  
The blocks start at $x,y,z,w$ respectively and their orientation 
is given by the arrow near to each starting point.
Then we can concatenate the even blocks (see top right picture) and the  odd blocks (see bottom right picture), 
obtaining two partially directed self-avoiding path.
}
\label{fig:IPDRW10}
\end{figure}

\section{Proof of Theorem \ref{thh1}}\label{proof of thh1} 

%

To prove Theorem \ref{thh1}, we show that for any $\beta \geq 0$ the partition function of
 IPSAW can be bounded from below and from above by the partition function of NE-IPSAW, by paying at most a sub-exponential price, i.e.,
\begin{equation}
\label{eq:samePF}
\mathrm{Z}_{\beta, L}^{\tx{NE}}\leq \mathrm{Z}_{\beta , L}\leq e^{o(L)}\,\mathrm{Z}_{\beta , L}^{\tx{NE}}, \quad \forall\, L\in  \mathbb N,\, \forall\, \beta \in [0,\infty).
\end{equation}
Where the $o(L)$ depends on $\beta$.

The lower bound in \eqref{eq:samePF} is trivial because NE-paths are a particular subclass of prudent paths. The proof of the upper bound is harder and needs some work. 
%
In a few word, 
we will apply a strategy which  consists, for every $L\in \N$ and $\beta\in (0,\infty)$, in building a mapping $M_L: \Omega_L^{\tx{PSAW}}\to \Omega_{L}^{\tx{NE}}$ which satisfies
the following conditions: 
\begin{enumerate}\label{constru}
\item \label{item:ML1} There exists a real function $f_1$ such that $| (M_L)^{-1}(\hat w)| \leq e^{f_1(L)},$  where $f_1(L)$ is uniform in  $\hat w\in \Omega_{L}^{\tx{NE}}$ and $f_1(L)=o(L)$,
\item \label{item:ML2} There exists a real function $f_2$ such that $H\,(w)-H\,\big(M_L(w)\big)\leq  f_2(L),$ where $f_2(L)$ is uniform in $w\in \Omega_L^{\tx{PSAW}}$ and $f_2(L)=o(L)$.
\end{enumerate}
The existence of $(M_L)_{L\in \N}$ satisfying the aforementioned properties is sufficient to prove  the upper bound in \eqref{eq:samePF}. The dependency in $\beta$ is dropped for simplicity.

We will build the mapping with the help of the macro-block decomposition of every path $w\in \Omega_L^{\tx{PSAW}}$ (recall Section \ref{sec:IPSAWdef}). By a succession of systematic 
transformations we will indeed map each macro-block onto an associated NE-macro-block in such a way that the resulting 
NE-macro-blocks can be concatenated into a NE-prudent path which will be the image of $w$ by $M_L$. Then, it will be enough to check that $(M_L)_{L\in \N}$ satisfies the aforementioned properties.

The first property, \eqref{item:ML1}, will be rigorously proven below and it is mostly a consequence of Lemma \ref{lemma:macroblocknumber} which states that the 
macro-block number is at most $\mathcal O (\sqrt{L})$. The second property, \eqref{item:ML2}, is the hardest to check. 
On the energetic point of view, the main difference between a generic prudent paths
and their North-East counterpart is that generic paths undergo interactions between macro-blocks. Such interactions turn out to be tuned by the first stretches of each macro-blocks.
Moreover, Lemma \ref{lemma:macroblocknumber} 
implies that an important loss between $w$ and $M_L(w)$ can only be observed when those first stretches are very large.
This is the reason why we remove such stretches from the path as soon as they are larger than a prescribed size,
 e.g., $L^{1/4}$. This only triggers
a sub-exponential loss of entropy since those large stretches are at most $L^{3/4}$. It might cause a large loss of energy, but this loss will be compensated by the construction of a large square block (i.e., maximizing the energy) 
containing all those stretches that we have removed.  

\medskip

We now start with the precise construction of $M_L$. For such purpose, we define $4$ sequences of applications that are mapping trajectories onto other trajectories. To be more specific, for every $L\in \N$,
 we define $5$ sets of trajectories $\mathcal W_{i,L},\, i=1,\dots, 5$, interpolating $\Omega_L^{\tx{PSAW}}=\mathcal W_{1,L}$ with $\Omega_L^{\tx{NE}}=\mathcal W_{5,L}$, and $4$ sequences of applications $\psi_L^{i}:\mathcal W_{i,L}\to\mathcal W_{i+1,L}$, cf. Steps 1-4 below.
We define $M_L$ as the composition of such maps $\psi_L^4,\dots,\psi_{L}^1$, i.e., $M_L:=\psi_L^4\circ\psi_L^3\circ\psi_L^2 \circ \psi_{L}^1$. 
To prove property \eqref{item:ML1} we show that each $\psi_L^i$ is \emph{sub-exponential}, i.e,
\begin{definition}\label{subex}
The sequence of mappings  $(\psi_L)_{L\in \N}$, with $\psi_L\colon \mathcal{W}_{1,L}\to \mathcal{W}_{2,L}$, is \emph{sub-exponential} if there exist $c_1,c_2\in (0,\infty)$ and 
$\alpha\in [0,1)$ such that for every $L\in \N$ and every $w\in \mathcal{W}_{2,L}$
\begin{equation}\label{defsub}
|(\psi_L)^{-1}(w)| \leq c_1 e^{c_2 L^\alpha }.
\end{equation} 
\end{definition}
In \emph{Step $5$} we complete the proof by showing that such $M_L$ satisfies also the second property \eqref{item:ML2}.

\medskip

\subsection{Step 1}\label{item:s2}
Let $w\in \Omega_L^{\tx{PSAW}}$ be a prudent path. We can decompose $w$ into a sequence of macro-blocks, 
$\Lambda=(\Lambda_1,\dots,\Lambda_m)$, where $m=m(w)\in\mathbb N$, cf. \eqref{eq:OmegaPsaw} and Section \ref{sec:IPSAWdef}. We observe that 
each macro-block $\Lambda_i\in \Omega_{L_i}^{x_i}$, with $x_i\in \{\text{NE, NW, SE, SW}\}$ and 
$L_i\in \N$ such that $L_1+\dots+L_m=L$.
Each macro-block $\Lambda_i$ can be decomposed into a sequence of blocks $(\pi_1^i,\dots,\pi_{r_i}^i)$, cf. Section
 \ref{sec:NE-IPSAW}. We stress that both such decompositions are uniques. 
For every $i=1,\dots, m$, we consider separately the subsequence of blocks with odd indices, i.e.,  $\pi^{(\textbf o), i}:=(\pi_k^i)_{k\in \{1,\dots, r_{i}\}\cap (2\N -1)}$ and the subsequence of blocks with even indices, i.e., $\pi^{(\textbf e), i}:=(\pi_k^i)_{k\in \{1,\dots, r_i\}\cap 2\N}$. We apply to each of them  
the following procedure (1-4), drawn in Figure \ref{fig:Step2}. In the sequel, this procedure will be referred to as the \emph{large stretches removing procedure}.
\begin{enumerate}
\item \label{item:step1.1} We consider the first macro-block $\Lambda_1$ and the odd block subsequence, 
$\pi^{(\textbf o), 1}=(\pi_k^1)_{k\in \{1,\dots, r_1\}\cap (2\N -1)}$. 
We start by considering the first stretch of the first block, $\pi_1^1$. 
   If this stretch is not larger than $L^{1/4}$
we stop the procedure for the subsequence $\pi^{(\textbf o), 1}$ and
we jump to (2). 
    Otherwise, if the first stretch is larger than $L^{1/4}$,  we pick it off, and we 
    reapply the procedure to the next stretch of the block. 
    
    It may be that the procedure leads to removing all the stretches in the first block. 
In such case we re-apply the same procedure to the next block of $\pi^{(\textbf o), 1}$  and so on,  
until we find the first stretch smaller than $L^{1/4}$. 
For instance, in the odd subsequence, if we have entirely removed the first block, 
then we re-apply the procedure to the third block. 
  If none of the stretches in the subsequence $\pi^{(\textbf o), 1}$ is smaller than $L^{1/4}$, 
  then the whole subsequence of blocks is removed and we stop the procedure for the subsequence.
  
\item We apply the procedure (1) to the even block subsequence,
  $\pi^{(\textbf e), 1}=(\pi_k^1)_{k\in \{1,\dots, r_1\}\cap 2\N }$, i.e., we start with the procedure (1) by considering the first stretch of the second block, $\pi_2^1$.

 \item We apply the procedure (1) to the very last block of the macro-block $\Lambda_1$ (if it has not been already 
 modified). 
 
  We will see in Step $3$ below the importance of applying the large-stretch removing procedure to the very last block.
 
  \item We repeat (1-3) for the macro-blocks $\Lambda_2,\dots,\Lambda_m$.
\end{enumerate}

\begin{figure}
\includegraphics[scale=0.7]{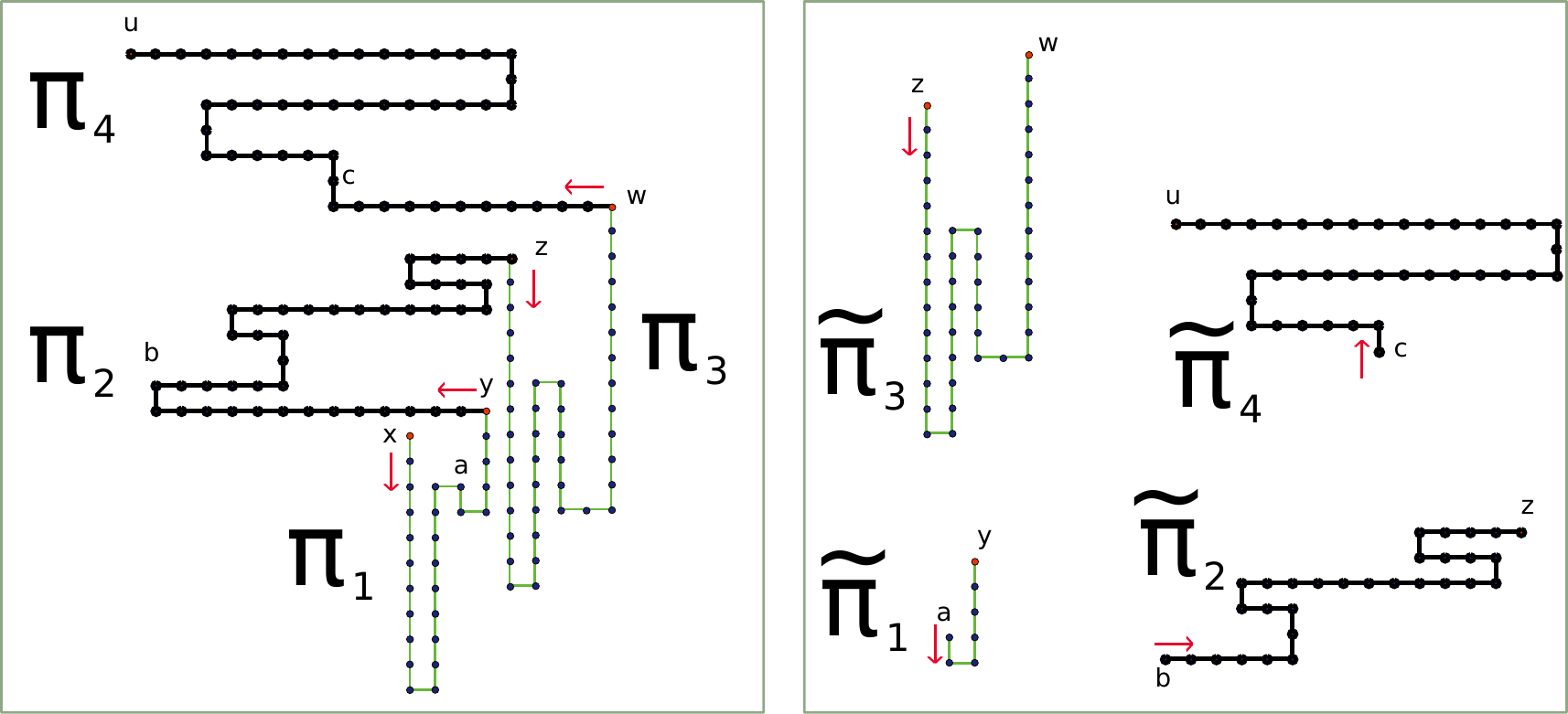}
\caption{A {NE}-prudent path decomposed into $4$ blocks $(\pi_1,\pi_2,\pi_3,\pi_4)$. 
We apply  the large-stretch removing procedure. The first $2$ stretches of $\pi_1$ are longer than $L^{1/4}$, therefore
we pick them off. The third stretch is smaller than $L^{1/4}$ and thus we stop the procedure on the odd subsequence.
We apply  the large-stretch removing procedure to the even subsequence. In this case we remove only the
first stretch of $\pi_2$ and we stop the procedure. 
Since $\pi_4$ is the last block of the trajectory we re-apply the large-stretch removing procedure to $\pi_4$. Also in this case
we remove only the first stretch. The result is the block sequence $(\tilde\pi_1,\tilde\pi_2,\tilde\pi_3,\tilde\pi_4)$. 
}
\label{fig:Step2}
\end{figure}

  \begin{remark}
  \label{reMinItem1}
 We note that picking off stretches does not change the exit condition, cf. Definition \ref{def:exit}. 
 To be more precise, given an oriented block with $N$-stretches, $(\ell_1,\dots,\ell_N)$, if we remove the first $k$-stretches ($k< N$), then the path obtained by 
 concatenating $(\ell_{k+1},\dots,\ell_N)$ still satisfies the same exit condition. The exit condition indeed means that $\ell_1+\dots+\ell_N > \max\{0, \ell_1,\dots,\ell_1+\dots+\ell_{N-1}\}$ and therefore $\ell_{k+1}+\dots+\ell_N > \max\{0, \ell_{k+1},\dots,\ell_{k+1}+\dots+\ell_{N-1}\}$. However, picking off stretches can change the initial condition of a block, 
 it could happen that the first stretch of the modified block is positive, i.e., $\ell_{k+1}\geq 0$. 
 \end{remark}
 
 At this stage, we need to give a mathematical definition of the large stretch removing procedure. To that aim, for every $L\in \N$,
 we denote by  $\psi_L^1:\Omega_L^{\tx{PSAW}}\to \psi_L^1\big(\Omega_L^{\tx{PSAW}}\big) $ the map that
 realizes the large stretches removing procedure.
 At the end of the present section, we will show that $(\psi_L^1)_{l\geq 1}$ is sub-exponential. However, for the sake of conciseness, the fine details of the proof will 
 be displayed only in the case for which we do not reapply the large stretch removing procedure to modify 
 the very last block of each macro-block. The proof in that case is very similar, see Remark \ref{rem:lastblock} below. 


\subsubsection{Large stretch removing procedure in a single macro-block}\label{singmac}

We pick $l\in \N$ and an orientation  $x\in \{\text{NE, NW, SE, SW}\}$. In the present section, we define the large stretch removing procedure on 
those macro-blocks in $\Omega_l^x$. To that aim, we define with (\ref{eq:defat0}--\ref{eq:PitileDef}) an application 
$\mathcal{T}_{l,L}: \Omega_l^x\mapsto \Omega^{l,x}_{\leq,L}$ that performs Procedure \eqref{item:step1.1}, i.e.,
removes the large stretches in a  single macro-block.   A rigorous definition of the image set
$\Omega^{l,x}_{\leq,L}$ will be provided in Definition \ref{def:omegaOH} below.

Before defining $\mathcal{T}_{l,L}$, let us briefly recall that we can associate with any arbitrary  macro-block
$\lambda\in \Omega_l^x$ an unique block sequence
$(\pi_1,\dots,\pi_r)$, with $r=r(\lambda)$. In particular it holds that 
$\lambda =\pi_1 \oplus\dots\oplus\pi_r$, see Section \ref{sec:NE-IPSAW}. Therefore, in the rest of the
section, we identify the macro-block with its block decomposition, i.e., $\lambda=(\pi_1,\dots,\pi_r)$.
For every $i\in \{1,\dots,r\}$, we let $N_i=N_i(\lambda)$ be the number of stretches
 in the $i$-th block (thus, cf. Section \ref{sec:IPDSAW}, the number of inter-stretches is $N_i-1$), and we let $(\ell_1^{(i)},\cdots, \ell_{N_i}^{(i)})$ be the sequence of stretches in the $i$-th block. 
 Since the sequence of stretches identifies the block, with a slight abuse of notation, we write $\pi_i=(\ell_1^{(i)},\cdots, \ell_{N_i}^{(i)})$.
The sequence of blocks $(\pi_1,\dots,\pi_r)$ can be partitioned into two subsequences 
$\mathbf{\pi^{(o)}}=(\pi_i)_{i\in \{1,\dots, r\}\cap (2\mathbb N -1)}$ and $\mathbf{\pi^{(e)}}=(\pi_i)_{i\in \{1,\dots, r\}\cap 2\mathbb N}$.

At this stage, we are ready to introduce the specific notations for the large stretches removing procedure. 
We let $k_1,k_2=k_1(\lambda),k_2(\lambda)\in \{1,\dots,r\}$  be the indices of the last block modified by the
large stretches removing procedure in the odd subsequence and in the 
even subsequence respectively (cf. \eqref{item:step1.1}). Analogously, let $j_1=j_1(\lambda)\in \{0,\dots,N_{k_1}\}$ and $j_2=j_2(\lambda)\in \{0,\dots,N_{k_2}\}$ be the index of the last stretch we removed in $\pi_{k_1}$ and $\pi_{k_2}$ respectively. 
By definition of  $r,k_1,k_2,j_1,j_2, N_m$ it holds that (note that the $\lambda$ dependency is dropped for simplicity)
\begin{align}\label{eq:defat0}
&|\ell_n^{(m)}| > L^{1/4}, \quad \text{for} \quad m\in\{1,\dots,k_1-1\}\cap (2\mathbb N-1),\  n\in \{1,\dots,N_m\},  \\
\nonumber &|\ell_n^{(k_1)}| > L^{1/4}, \quad \text{for} \quad n\in \{1,\dots,j_1\}, \\
\nonumber &|\ell_{j_1+1}^{(k_1)}| \leq L^{1/4}; 
\\ \nonumber \\
\label{eq:defat0bis} &|\ell_n^{(m)}| > L^{1/4}, \quad \text{for} \quad m\in\{1,\dots,k_2-1\}\cap 2\mathbb N,\  n\in \{1,\dots,N_m\}, \\
\nonumber &|\ell_n^{(k_2)}| > L^{1/4}, \quad \text{for} \quad n\in \{1,\dots,j_2\},\\
\nonumber &|\ell_{j_2+1}^{(k_2)}| \leq L^{1/4}.
 \end{align}

We let  $\mathcal{T}_{l,L}(\lambda)$
be the sequence of blocks remaining once the large stretch removing procedure in the macro-block $\lambda$ is complete. To be more specific, the subsequence of odd blocks $(\mathcal{T}_{l,L}(\lambda)_i)_{i\in \{1,\cdots, r\}\cap (2\mathbb N-1)}$ is defined as
 \begin{equation}
 \label{eq:PitileDef}
 \begin{split}
 & \mathcal{T}_{l,L}(\lambda)_k=\emptyset, \quad \forall\, k\in\{1,\dots,k_1-1\}\cap (2\mathbb N-1),  
 \\
& \mathcal{T}_{l,L}(\lambda)_{k_1}=\Big(\ell_{j_1+1}^{(k_1)},\dots,\ell_{N_{k_1}}^{(k_1)}\Big),  
 \\
& \mathcal{T}_{l,L}(\lambda)_{k}=\pi_k,\quad \forall\, k\in\{k_1+1,\dots,r\}\cap (2\mathbb N -1).
 \end{split}
 \end{equation}
 The subsequence of even blocks $(\mathcal{T}_{l,L}(\lambda)_i)_{i\in \{1,\dots, r\}\cap 2\mathbb N}$ is defined in the same way. 
 
 
%

 \begin{remark}
 \label{rem:StIt3}
We stress that if we start with a sequence of blocks 
 $\lambda=(\pi_1,\dots,\pi_r)\in  \Omega_{l}^{x}$, then, in general, it is not true that the sequence $\mathcal{T}_{l,L}(\lambda) = (\mathcal{T}_{l,L}(\lambda)_1,\dots,\mathcal{T}_{l,L}(\lambda)_r)$ we
defined in \eqref{eq:PitileDef} is still a decomposition of a $x$-prudent path, i.e., 
$\mathcal{T}_{l,L}(\lambda)$ may not belong to $\Omega_{s}^{x}$, for any $s\leq l$.  
For this reason we define here below a new set of oriented paths, $\Omega_{\leq,L}^{l,x}$, which gathers the 
images of all paths in $\Omega_l^x$ through $\mathcal{T}_{l,L}$.
 \end{remark}

 \begin{definition}\label{def:omegaOH} 
We say that a block sequence $\lambda=(\pi_1,\dots,\pi_r), \, r\in \{0,\dots,L\}$ belongs to $\Omega_{\leq,L}^{l,x}$  if and only if 
\begin{itemize}
\item $r\leq L$ and there exists $k_1\in 2\N-1$ and $k_2\in 2\N$ such that $k_1,k_2\leq \max\big\{r,\frac{l}{L^{1/4}}\big\}$ and 
$\pi_{i}=\emptyset$ for $i\in \{1,\dots,k_1-2\}\cap 2\N-1$ and for $i\in \{1,\dots,k_2-2\}\cap 2\N$, whereas 
$\pi_{i}\neq\emptyset$ for $i\in \{k_1,\dots,r\}\cap 2\N-1$ and for $i\in \{k_2,\dots,r\}\cap 2\N$.
\item the $x$ orientation is respected (cf. Section \ref{sec:NE-IPSAW}), e.g., 
in the case of $x=\tx{NE}$, then, every 
$\pi_i$ with ${i\in\{k_1,\dots, r\}\cap (2\mathbb N - 1)}$ is south-north (resp. west-east) and every 
$\pi_i$ with ${i\in\{k_2,\dots, r\}\cap 2\mathbb N }$ is west-east  (resp. south-north).  

\item There is no restriction on the orientation and on the length of the first stretch of $\pi_{k_1}$ and $\pi_{k_2}$. 

\item The total length (the sum of the length of every stretches in $(\pi_1,\dots,\pi_r)$) is smaller than $l$.
\end{itemize}
\end{definition}

We conclude this section with the computation of an upper bound on the cardinality of the 
ancestors of an arbitrary $\gamma\in \Omega_{\leq, L}^{l,x}$ by $\mathcal{T}_{l,L}$. 
We denote by  $h$ the total length of $\gamma$. Let $\lambda\in \Omega_l^x$ be an ancestor of 
$\gamma$ by $\mathcal{T}_{l,L}$.  The total length of those stretches removed from 
$\lambda$ by $\mathcal{T}_{l,L}$ to get $\gamma$ necessarily equals $l-h$. By definition, cf. \eqref{eq:PitileDef}, the number of empty blocks in $\gamma$ is
$k^{'}_1:=\frac{k_1-1}{2}$ (resp. $k^{'}_2:=\frac{k_2-2}{2}$) for the odd subsequence  (resp. for the even subsequence) of blocks.
Therefore, since $\mathcal T_{l,L}$ may remove only stretches larger than $L^{1/4}$, the number $v$ of stretches removed from $\lambda$ to get $\gamma$ satisfies $k^{'}_1+k^{'}_2 + 2\leq v\leq (l-h)/L^{1/4}$. This suffices to write the following upper bound 
\begin{equation}\label{uppbsb}
\Big| (\mathcal{T}_{l,L})^{-1}(\gamma)\Big |\leq \sum_{v=k^{'}_1+k^{'}_2+2}^{(l-h)/L^{1/4}} 2^v \, \binom{l-h}{v}\,  \binom{v}{k^{'}_1+k^{'}_2+2}.
\end{equation}
The summation in \eqref{uppbsb} runs over $v$ which stands for the number of stretches removed from $\lambda$. 
Let us explain \eqref{uppbsb}. Once $v$
is chosen, reconstructing $\lambda$ requires to choose the length of each removed stretches and these choices are less than  the binomial factor $\binom{l-h}{v}$. Once, the length of each removed stretch is chosen, one has to chose their orientations 
which gives at most $2^v$ choices.
Finally, those deleted stretches have to be distributed among the  $k^{'}_1+k^{'}_2+2$ blocks in $\gamma$  that have  to be completed by other stretches to recover $\lambda$.  This gives rise to the term  $\binom{v}{k^{'}_1+k^{'}_2+2}$. Then,
the fact that $k^{'}_1+k^{'}_2+2\leq (l-h)/L^{1/4}$ allows us to bound from above the r.h.s. in \eqref{uppbsb} by 
\begin{equation}\label{endup}
\Big| (\mathcal{T}_{l,L})^{-1}(\gamma)\Big |\leq e^{c_0 l \log(L)/L^{1/4}},
\end{equation}
for some constant $c_0\in (0,\infty)$.

\subsubsection{Large stretch removing procedure for a generic prudent path.}\label{mulb}
We are ready to define the map $\psi_{L}^1$, which defines the large stretch removing procedure applied to generic prudent path. 
We recall equation   \eqref{eq:OmegaPsaw}, which asserts that a path $w\in  \Omega_L^{\tx{PSAW}}$
can be  decomposed into $m=m(w)\in \N$ macro-blocks  $(\Lambda_1,\dots,\Lambda_m)$.
Such macro-block decomposition is an element of $\Theta_{m,L}$ and each macro-block
belongs to some $\Omega_{t_i}^{x_i}$ (see \eqref{THETAml}) with $t_1+\dots+t_m=L$. Thus, we define $\psi_L^1$ by applying, for every $i\leq m$,
the map $\mathcal{T}_{t_i,L}$ to $\Lambda_i$, i.e., 
\begin{equation}\label{trhu}
\psi_L^1(w):= \big( \mathcal T_{t_1,L}(\Lambda_1),\dots, \mathcal T_{t_m,L}(\Lambda_m)   \big).
\end{equation}
 The image set of $\Omega_L^{\tx{PSAW}}$ by $\psi_L^1$ is 
 is therefore $
 \mathcal{W}_{2,L}:=\bigcup_{m\in \N} \psi_L^1\big(\Theta_{m,L}\big)$
 which is a subset of 
%
%
\begin{equation}  \label{eq:W2L}
 \bigcup_{m \in \mathbb N}\bigcup_{L_1+\dots+L_m=L}
 \bigcup_{\substack{(x_i)_{i=1}^m\in \{\tx{NE}, \tx{NW}, \tx{SE}, \tx{SW}\}\\
x_{i-1}\neq x_{i}}} \, 
\Omega_{\leq,L}^{L_1,x_1}\times\dots\times\Omega_{\leq,L}^{L_m,x_m}.
 \end{equation}
 Let us observe that the union over $m$ is finite, because, by Lemma \ref{lemma:macroblocknumber}, the number of macro-blocks $m$ is  at most $c L^{1/2}$, for some universal constant $c\in(0,\infty)$. 
 Moreover, let us observe that \eqref{eq:W2L} is not a disjoint union.

The step will be completed once we show that $\psi_L^1$ is sub-exponential.
To that aim, we need an upper bound on the 
 cardinality of $(\psi_L^1)^{-1}(\tilde \Lambda)$ that is  uniform on the choice of $\tilde \Lambda\in \psi_L^1\big(\Omega_L^{\tx{PSAW}}\big)$. Thus, we pick $\tilde \Lambda \in \psi_L^1\big(\Omega_L^{\tx{PSAW}}\big)$ and we consider its macro-block decomposition 
 $(\tilde \Lambda_1,\dots,\tilde \Lambda_m)$.  
Before counting the number of ancestors of $\tilde \Lambda$ by $\psi_L^1$, one should note that 
$\tilde \Lambda$ may belong to more than one set of the form $\Omega_{\leq, L}^{L_1,x_1}\times\dots\times\Omega_{\leq,L}^{L_m,x_m}$. However, since $m=\mathcal O(L^{1/2})$ 
(cf. Lemma \ref{lemma:macroblocknumber}) and since $L_1+\dots+L_m=L$, the number of such sets is bounded from above by 
$\sum_{m=1}^{c\sqrt{L}}\binom{L}{m}$, for some $c \in (0,\infty)$. This quantity is less than $c\sqrt{L}\binom{L}{c\sqrt{L}}\leq e^{2 c\sqrt{L} \log(L)}$.  It remains to count the number of ancestors of 
$\tilde \Lambda$ within  a given $\Omega_{L_1}^{x_1}\times\dots\times\Omega_{L_m}^{x_m}$. By \eqref{endup} above,
this is at most  $e^{c_0L_1 \log(L)/L^{1/4}}\times \dots \times  e^{c_0L_m \log(L)/L^{1/4}}$ which again is smaller 
than $e^{c_0 L^{3/4}\log(L)}$. This suffices to conclude that $\psi_L^1$ is sub exponential.

 \begin{remark}
\label{rem:lastblock}
When we prove that  $\psi_L^1$ is sub exponential, we have not taken into account the fact that the large stretch removing procedure 
should also be applied to the very last block of each macro-block. However, this affects only marginally our computations and does not 
modify the sub-exponentiality of $\psi_L^1$.
To be more precise, if we also modify the very last block in any macro-block, then to bound from above the
 number of ancestors of $\tilde \Lambda$ by $\psi_L^1$, we 
consider separately two parts. In the first part, we apply
 the large stretches removing procedure to each macro-block without consider the 
very last block of any macro-block. This part has been already considered in the discussion above, which gave rise to
 \eqref{uppbsb} and \eqref{endup}. Then we consider the large stretches removing procedure apply only to any last block of any macro-block. It is not difficult to check that \eqref{uppbsb} provides an upper bound also for this part of the 
 procedure.
 Therefore, we conclude that also in this general case \eqref{endup} still holds up to a constant.
%
\end{remark}


\subsection{Step 2}
 \label{item:s3} 
 In Step 1 we considered $w\in \Omega_L^{\tx{PSAW}}$ and we decomposed it into a sequence of macro-blocks, cf. \eqref{eq:OmegaPsaw}, $\Lambda=(\Lambda_1,\dots,\Lambda_m)$, where $m=m(w)\in\mathbb N$. We let $(\tilde\Lambda_1,\dots,\tilde\Lambda_m)=\psi_L^1(w)$ be the result of the large stretch removing procedure. 
 Each $\tilde\Lambda_i$ is defined by a sequence $(\tilde\pi_1^i,\dots,\tilde\pi_{r_i}^i)$ which is not necessary concatenable, cf. Remark \ref{rem:StIt3} and Section \ref{decomp}.
In this step we aim at modifying all the sequences $(\tilde\pi_1^i,\dots,\tilde\pi_{r_i}^i)$, for $i=1,\dots, m$, in order to recover a concatenable block sequence. In the sequel this procedure will be referred to as the \emph{concatenating block procedure}. 

Our procedure $\psi_L^2$ acts on  $\mathcal{W}_{2,L}$ (recall \eqref{eq:W2L}). To be more specific, $\psi_L^2$ takes as an argument 
an element  
$$\tilde \Lambda=(\tilde \Lambda_1,\dots,\tilde \Lambda_m)\in \Omega_{\leq,L}^{L_1,x_1}\times\dots\times\Omega_{\leq,L}^{L_m,x_m}$$
where  $m\leq cL^{1/2}$, where  $(L_1,\dots,L_m)$ is a sequence of length such that $L_1+\dots+L_m=L$, where $(x_1,\dots,x_m)$ is a sequence of orientations and where we keep in mind that $\tilde \Lambda$ is in the image set of $\psi_L^1$ . As a result, $\psi_L^2$ provides us with  a 
sequence of macro-blocks 
$$\psi_{L}^2(\tilde \Lambda)=\hat \Lambda=(\hat \Lambda_1,\dots,\hat \Lambda_m)$$ 
where, for every $i\leq m$,  $\hat \Lambda_i\in \Omega_{t_i}^{x_i}$ with  $t_i$  the total length of $\tilde \Lambda_i$.

We describe the procedure on a single modified macro-block $\tilde\Lambda$ in Section
\ref{concatonesin} below.  Later on, we generalize the procedure to the whole block-sequence in Section \ref{concatonin}.

\smallskip 

\subsubsection{Concatenating block procedure in a single macro-block}
\label{concatonesin}

We pick $h\leq l\in \N$ and consider   $\tilde \lambda=(\tilde\pi_1,\dots,\tilde\pi_{r})\in \Omega^{l,x}_{\leq,L}$ such that the total length of 
$\tilde \lambda$ equals $h$.

\smallskip

 %
Recall the definition of $k_1(\tilde \lambda)$ and $k_2(\tilde \lambda)$ in  Definition \ref{def:omegaOH}. By Remark \ref{reMinItem1} it turns out that $\tilde \lambda$ 
fails to be concatenable only if $|k_1-k_2|\geq 3$ that is  if
there exists an $i\leq r$ such that $\tilde\pi_i,\tilde\pi_{i+2}\neq \emptyset$ and $\tilde\pi_{i+1}=\emptyset$. In such case indeed, if  the last stretch of $\tilde\pi_i$ and the first stretch of $\tilde\pi_{i+2}$ have opposite orientations (see Figure \ref{fig:Step3}) then $\tilde \pi_i$ and $\tilde \pi_{i+2}$
are not concatenable.
Making  $\tilde \pi_i$ and $\tilde{\pi}_{i+2}$ concatenable possibly requires to slightly modify their structure. 
To be more specific, if the first stretch of $\tilde\pi_{i+2}$ and/or the last stretch of 
$\tilde\pi_i$ have zero length, then $\tilde\pi_{i+2}$ and 
$\tilde\pi_{i}$ are always concatenable. 
In this case we do not need to change their structure to make them concatenable. 
Otherwise, if the first stretch of $\tilde\pi_{i+2}$ has non-zero length, then
it is always possible to modify the first step in the first stretch of $ \tilde\pi_{i+2}$ to transform it into an inter-stretch, see Figure \ref{fig:Step3}, and after this simple transformation $ \tilde\pi_{i}$ and $ \tilde\pi_{i+2}$ become always concatenable.  
Thus, in the case where $k_1\leq k_2-3$ (the case $k_2\leq k_1-3$ is similar) it suffices to apply the aforementioned transformation to each 
blocks $\tilde \pi_{k_1+2}, \dots, \tilde \pi_{k_2-1}$ and  to concatenate $\tilde \pi_{k_1}, \dots, \tilde \pi_{k_2-1}$ into a unique oriented block, say $\hat \pi_1'$. We remove those empty blocks $\tilde\pi_{i}$ indexed in $ \{1,\dots,k_1-2\}\cap 2\N-1$ and in  $\{1,\dots,k_2-2\}\cap 2\N$
to get finally  the concatenable sequence $(\hat \pi_1',\tilde\pi_{k_2},\dots,\tilde \pi_r)$. The path 
$\hat \lambda :=\hat\pi_1'\oplus\tilde\pi_{k_2}  \oplus\dots \oplus  \tilde\pi_r\in \Omega_{h}^{x}$.

\begin{remark}
\label{remRelou}
It is important to keep in mind that the concatenable sequence $(\hat \pi_1',\tilde\pi_{k_2},\dots,\tilde \pi_r)$ is not a standard decomposition 
of a NE-prudent path, cf. Definition \ref{def:NEPrudentPath}: in this case we do not have any constriction 
on the first stretch of  $\tilde \pi_{k_2}$ and $\tilde\pi_r$ 
(if the last block was changed by the large stretches removing procedure) 
other than to be smaller than $L^{1/4}$, cf. Remark \ref{reMinItem1}. 
It is necessary to slightly redefine $\hat \pi_{1}'$ and $\tilde \pi_{k_2}$ in order to obtain two proper oriented 
 blocks, say $\hat \pi_{1}$ and $\hat \pi_{2}$. 
 We also modify $\tilde\pi_{r-1}$ and $\tilde\pi_r$ in the same way
 to obtain the oriented blocks $\hat\pi_{s-1}$ and $\hat\pi_s$, where $s=(k_1+k_2)/2-2$. We observe that
 we can do this modification to have that $\hat\pi_s\subseteq \tilde\pi_r$.
 In such a way
 the block sequence $(\hat\pi_1,\dots,\hat\pi_s)$ is a proper decomposition of a NE-prudent path.
 We observe that a very crude bound tells us that the number of ancestors of a block by this last transformation 
is bounded above by its total number of stretches, which is smaller than $l$. 
\end{remark} 

\begin{remark}
In principle, if the last stretch of $\tilde\pi_i$ and the first stretch of $\tilde\pi_{i+2}$ have both non-zero
length and the same orientation,
then it would be possible to concatenate $\tilde\pi_i$ with $\tilde\pi_{i+2}$. Anyway, also in this
case we modify the $\tilde\pi_{i+2}$ structure, as prescribed by the aforementioned transformation. 
We do that for computational convenience, as it will be clear in \eqref{uppbsb2} below.
\end{remark}

 The procedure described above corresponds to the mapping $\mathcal R_{l,L}: \Omega_{\leq,L}^{l,x}\mapsto \cup_{h\leq l} \, \Omega_{h}^{x}$.   As we did in Section \ref{singmac}, we need to conclude this section by computing, for $h\leq l\leq L$ and $x\in \{\text{NE, NW, SE, SW}\}$,  the number of ancestors in $\Omega_{\leq,L}^{l,x}$ of a given
$\gamma\in \Omega_h^x$ by $\mathcal R_{l,L}$. 
To that aim, we write  $\gamma:=(\hat \pi_1,\dots,\hat \pi_s) \in  \Omega_{h}^{x}$ and we consider  $\tilde \lambda=(\tilde \pi_1,\dots,\tilde \pi_r)\in \Omega_{\leq,L}^{l,x}$ an ancestor of 
$\gamma$ by $\mathcal{R}_{l,L}$. For simplicity, assume also that $k_1=k_1(\tilde \lambda)\leq k_2(\tilde{\lambda})=k_2$ and recall that, 
by Definition \ref{def:omegaOH}, we have necessarily $k_1,k_2\leq \frac{l}{L^{1/4}}$. Thus, we have necessarily that all blocks 
$(\tilde \pi_1,\tilde \pi_3,\dots,\tilde \pi_{k_1-2})$ and all blocks 
$(\tilde \pi_2,\tilde \pi_4,\dots,\tilde \pi_{k_2-2})$ are empty. 
Moreover, we explained above that $\hat \pi_1$ is essentially obtained by modifying 
the first step of the first stretch of some oriented blocks in
$(\tilde \pi_{k_1}, \tilde \pi_{k_1+2},\dots,  \tilde \pi_{k_2-1})$.
 This suffices to write the following upper bound 
\begin{equation}\label{uppbsb2}
\Big| (\mathcal{R}_{l,L})^{-1}(\gamma)\Big |\leq \sum_{k_1,k_2\leq l/L^{1/4}} l\, 2^{|k_1-k_2|}\binom{l}{\frac{|k_1-k_2|}{2}},
\end{equation}
The summation in \eqref{uppbsb2} runs over $k_1,k_2$ which provides the number of empty blocks at the beginning of the odd and even sequences of blocks in $\tilde \lambda$ and, once $k_1$ and $k_2$ are chosen, one can reconstruct $(\tilde \pi_{k_1}, \tilde \pi_{k_1+2},\dots,  \tilde \pi_{k_2-1})$
from  $\hat  \pi_1$ by decomposing  $\hat  \pi_1$ into $(k_2-k_1)/2$ groups of consecutive stretches. This provides at most 
$\binom{l}{\frac{|k_1-k_2|}{2}}$ choices since the number of stretches in $\hat \pi_1$ is at most $l$.  
Then we have to take in account the transformation we made on  the first step of the first stretch of some 
oriented blocks in $(\tilde \pi_{k_1}, \tilde \pi_{k_1+2},\dots,  \tilde \pi_{k_2-1})$. This provide at most two
configuration for each such block and thus the factor $2^{|k_1-k_2|}$.
The factor $l$ is due to the fact that we have at most $l$ different way to choose $\tilde\pi_{k_{2}-1}$ and $\tilde\pi_{k_2}$ and $\tilde\pi_{r-1}$ and $\tilde\pi_{r}$, cf. Remark \ref{remRelou}.
At this stage, it is sufficient to 
recall that $k_2-k_1\leq l/L^{1/4}$ to rewrite \eqref{uppbsb2} as  
\begin{equation}\label{endup2}
\Big| (\mathcal{R}_{l,L})^{-1}(\gamma)\Big |\leq \frac{l^3}{L^{1/2}}\, 2^{l/L^{1/4}}\, e^{ l \log(L)/L^{1/4}}\leq e^{ c_1 l \log(L)/L^{1/4}},
\end{equation}
for some $c_1\in (0,\infty)$.

\subsubsection{Concatenating block procedure for a generic path}
\label{concatonin}

We are ready to define the map $\psi_{L}^2$  on those generic macro-block sequences from $\mathcal{W}_{2,L}$. 
We recall Definition \ref{eq:W2L}, we pick $m\leq c\sqrt{L}$ and $(L_1,\dots,L_m)\in \N^m$ satisfying 
$L_1+\dots+L_m= L$. Then, we pick 
$$\tilde \Lambda=(\tilde \Lambda_1,\dots,\tilde \Lambda_m)\in \Omega_{\leq,L}^{L_1,x_1}\times\dots\times\Omega_{\leq,L}^{L_m,x_m},$$
and we define $\psi_L^2$ by applying, for every $i\leq m$,
the map $\mathcal{R}_{L_i,L}$ to $\tilde \Lambda_i$, i.e., 
\begin{equation}\label{trhu}
\psi_L^2(\tilde\Lambda):= \big( \mathcal R_{L_1,L}(\tilde \Lambda_1),\dots, \mathcal R_{L_m,L}(\tilde \Lambda_m)   \big).
\end{equation}
The image set of $\mathcal{W}_{2,L}$ by $\psi_L^2$ 
 is therefore denoted by $\mathcal{W}_{3,L}$ and it is  a subset of 
\begin{equation}  \label{eq:W3L}
\bigcup_{m \leq c L^{1/2}}\, \bigcup_{l_1+\dots+l_m\leq L}\, 
 \bigcup_{\substack{(x_i)_{i=1}^m\in \{\tx{NE}, \tx{NW}, \tx{SE}, \tx{SW}\}\\
x_{i-1}\neq x_{i}}} \, 
\Omega_{l_1}^{x_1}\times\dots\times\Omega_{l_m}^{x_m},
 \end{equation}
 where the union over $m$ is truncated at $c L^{1/2}$ thanks to Lemma \ref{lemma:macroblocknumber}.

 \begin{remark}\label{remin}
Let us stress the fact that, as explained in Section \ref{mulb} above,  a given $\tilde \Lambda\in \mathcal W_{2,L}$ may well belong to 
more than one set  of the form $\Omega_{\leq,\tx{vh}}^{L_1,x_1}\times\dots\times\Omega_{\leq,\tx{vh}}^{L_m,x_m}$. This may be confusing because 
the definition of $\psi_L^2$ in \eqref{trhu} seems to depend on the choice of $L_1,\dots,L_m$. However, this is not the case because 
the applications $\mathcal {R}_{l,L}$ do actually not depend on $l$.
\end{remark}

The step will be complete once we show that $\psi_L^2$ is sub-exponential.
To that aim, we need an upper bound on the 
 cardinality of $(\psi_L^2)^{-1}(\hat \Lambda)$ that is  uniform on the choice of $\hat \Lambda\in \psi_L^2\big(\mathcal{W}_{2,L}\big)$. Thus, we pick $\hat \Lambda \in \psi_L^2\big(\mathcal W_{2,L}\big)$ and we consider its macro-block decomposition 
 $(\hat \Lambda_1,\dots,\hat \Lambda_m)$ which belongs to 
$ \Omega_{l_1}^{x_1}\times\dots\times\Omega_{l_m}^{x_m}$ for some $l_1+\dots+l_m\leq L$.
Before counting the number of ancestors of $\hat \Lambda$ by $\psi_L^2$, one should note that 
the ancestors of $\hat \Lambda$  may belong to any set of the form $\Omega_{\leq,L}^{L_1,x_1}\times\dots\times\Omega_{\leq, L}^{L_m,x_m}$
with $L_1+\dots+L_m\leq L$ and $L_i\geq l_i$ for every $i\leq m$. Again, since $m\leq c\sqrt{L}$, the number of such sets is bounded above by 
$\binom{L}{c\sqrt{L}}\leq e^{c\sqrt{L} \log(L)}$.
  It remains to count the number of ancestors of 
$\hat \Lambda$ within  a given $\Omega_{\leq, L}^{L_1, x_1}\times\dots\times\Omega_{\leq, L}^{L_m,x_m}$ and by \eqref{endup2} above,
this is at most  $e^{c_1 L_1 \log(L)/L^{1/4}}\times \dots \times  e^{c_1 L_m \log(L)/L^{1/4}}$ which again is smaller 
than $e^{c_1 L^{3/4}\log(L)}$. This suffices to conclude that $\psi_L^2$ is sub exponential.

\begin{figure}
\includegraphics[scale=0.6]{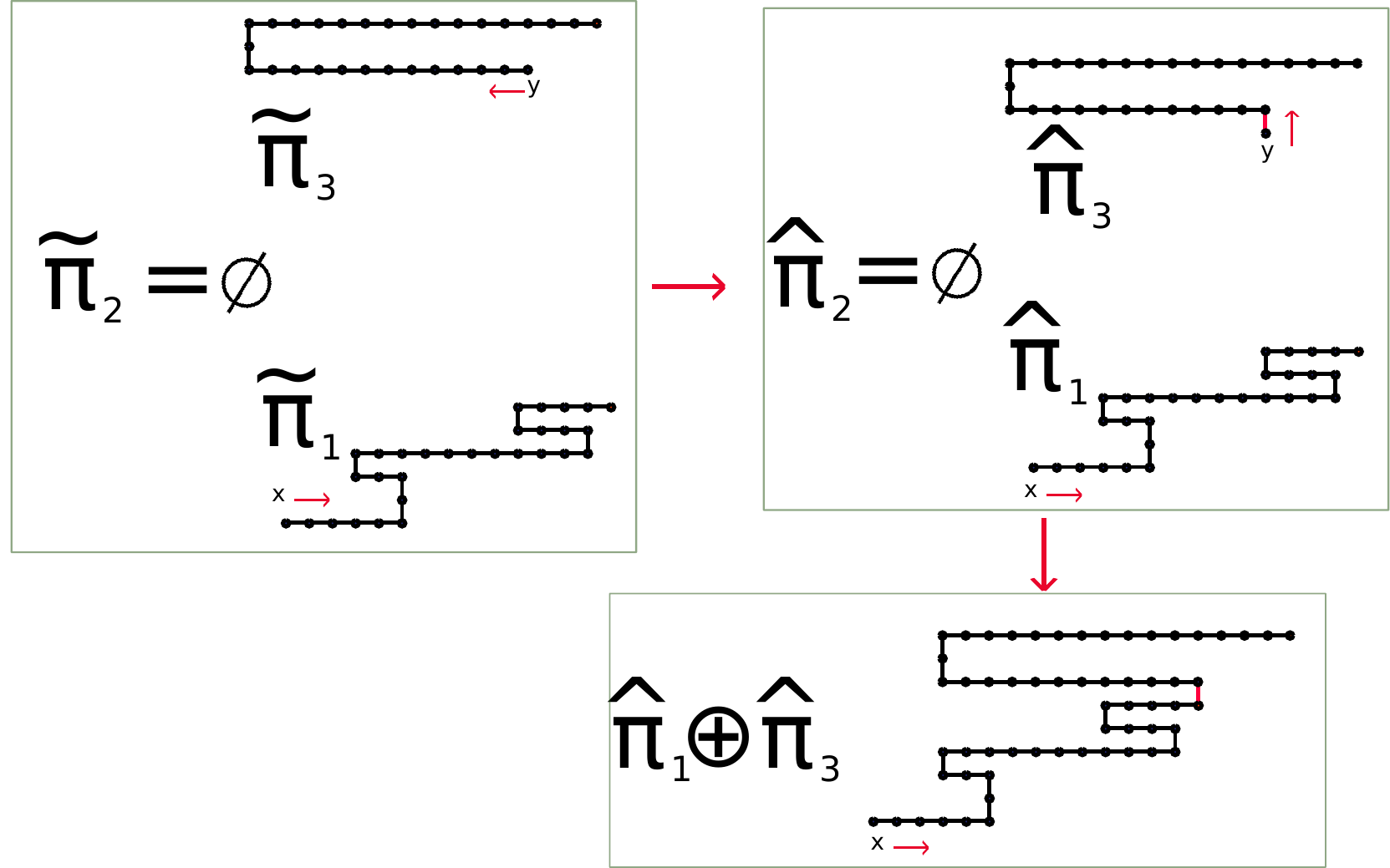}
\caption{We consider a sequence $(\tilde\pi_1,\tilde\pi_2,\tilde\pi_3)$ provided by the {large stretch removing procedure} in Step 2.
In this case we have that the {large stretch removing procedure} has removed the block $\tilde\pi_2$. 
 We modify the first step of the fist stretch of $ \tilde\pi_{3}$ in order to appear artificially an inter-stretch. 
 In such a way we can safely concatenate the blocks 
 $\hat\pi_{1}$ with $\hat\pi_{3}$ in a unique block $\hat\pi_1\oplus\hat\pi_3$.}
 \label{fig:Step3}
\end{figure}

\subsection*{Step 3}
\label{item:s4} 

In this step we consider a macro-block sequence $(\hat\Lambda_1,\dots\hat\Lambda_m)\in \mathcal W_{3,L}$ and we begin by modifying each macro-block $\hat\Lambda_i$ in order to recover a sequence of concatenable macro-blocks with only NE-orientations. Then we concatenate those modified north-east macro-blocks to recover a two sided path.
 In the sequel we refer to such procedures as
 \emph{macro-block concatenating procedure}.
 
 This procedure is defined through the function $\psi_L^3$, which acts on  $\mathcal{W}_{3,L}$ (recall \eqref{eq:W3L}). To be more specific, $\psi_L^3$ takes as an argument 
an element  
\begin{equation}\label{defll}
\hat \Lambda=(\hat \Lambda_1,\dots,\hat \Lambda_m)\in \Omega_{l_1}^{x_1}\times\dots\times\Omega_{l_m}^{x_m} .
\end{equation}
By keeping in mind that $\hat \Lambda$ is in the image set of $\psi_L^2(\psi_L^1)$, in \eqref{defll} $m\leq cL^{1/2}$ by Lemma \ref{lemma:macroblocknumber}, $(l_1,\dots,l_m)\in \N_{0}^m$ is a given 
integer vector such that $l_1+\dots+l_m\leq L$ and 
$(x_1,\dots,x_m)$ is a sequence of orientations. 
As a result, $\psi_L^3$ provides us with  a
north east prudent path of length $l_1+\dots+l_m$, i.e., an element of 
$\Omega^{\tx{NE}}_{l_1+\dots+l_m}$.
 
%
 
 \smallskip 

\subsubsection{Giving a macro-block a north-east orientation}
\label{reorient}
 In this section we pick $l\in \N$, $x$ an orientation and we consider 
$\hat \lambda=(\hat \pi_1,\dots,\hat \pi_r)\in \Omega_l^{x}$ a macro-block such that 
$\hat \pi_r:=(\hat \ell^{\, r}_1,\dots,\hat \ell^{\, r}_{N_r})$ either  satisfies the \emph{upper exit condition}, i.e., 
$\ell_1^{\, r}+\cdots+\ell^{\,r}_{N_r}\, >\, \max_{\, 0\leq i< N_r}\{\ell_1^{\, r}+\cdots+\ell_i^{\, r}\},$
or satisfies the  \emph{ lower exit condition}, i.e., 
$
 \ell_1^{\,r}+\cdots+\ell_{N_r}^{\, r}\, <\, \min_{\, 0\leq i< N_r}\{\ell_1^{\,r}+\cdots+\ell_i^{\, r}\}$ (we recall Definition \ref{def:exit}).

Giving a \emph{north-east} orientation to $\hat{\lambda}$ and making sure that it will be concatenable with 
other north east macro-blocks  
requires to perform $3$ transformations on each $\hat \lambda$. 
Among those $3$ geometric transformations, the first two are simple and the third is more involved and we will describe it carefully below.


To begin with, we recall Section \ref{sec:NE-IPSAW} and we observe that any \emph{two-sided} prudent path can be mapped onto 
a \emph{north-east} prudent path subject to at most two axial symmetries.  Therefore, we map $\hat \lambda$ onto $\hat \lambda_{\tx{NE}}$
and we note that at most $4$ ancestors can be mapped onto the same  
\emph{north-east} macro-block. For simplicity, we keep the notation  
$\hat \lambda_{\tx{NE}}=(\hat \pi_1,\dots,\hat \pi_s)$ and we note that $\hat \pi_s$  still satisfies 
either the  \emph{upper exit condition} or the  \emph{lower exit condition}. 
At this stage, we need to make sure that $\hat \lambda_{\tx{NE}}$ will be concatenable with other  \emph{north-east} macro-blocks. To that aim, 
we  follow the procedure described in Step 2, i.e., in case  $\hat \pi_1$ does not start by an inter-stretch ($\ell^1_1 \neq 0$) 
we modify the first step of its very first stretch, in such a way that this step becomes an inter-stretch.
This amounts to add a zero-length stretch at the beginning of $\hat \pi_1$ and to reduce the length of $\ell_1^1$ by one unit. By reasoning as in Step 2,
this second transformation maps at most two macro-blocks onto the same macro-block.

After these first two transformations, we can not yet claim that $\hat \lambda_{\tx{NE}}$ is concatenable with any other north-east 
macro-blocks. The macro-block $\hat \lambda_{\tx{NE}}$ is indeed concatenable if $\hat \pi_s$, the last oriented block of  $\hat \lambda_{\tx{NE}}$, satisfies the 
\emph{upper exit condition}, but we have seen that it may well satisfy the \emph{lower exit condition}. In this last case,
we need to apply a third transformation to $\hat \lambda_{\tx{NE}}$ to make sure that its last block satisfies the 
\emph{upper exit condition}. For this purpose we recall that $\hat\pi_{s-1}$ and $\hat\pi_s$ are obtained as a slight
modification of $\tilde\pi_{r-1}$ and $\tilde\pi_r$ and $\hat\pi_s\subseteq \tilde\pi_r$, cf. Section \ref{concatonesin} and Remark \ref{remRelou}. 
Moreover, we recall that 
 $\tilde\pi_r$ is the result of the the large stretch removing procedure applied to $\pi_r$,
  thus, the length of its first stretch is smaller than $L^{1/4}$. This ensures that 
there exists a partially directed path $\pi$ contained in $\hat\pi_{s-1}\cup \hat\pi_s$ and that contains $\hat\pi_1$ such that its first stretch is smaller than $L^{1/4}$. Moreover, $\pi$ has the same orientation of $\hat\pi_s$. For instance in Figure \ref{fig:FlipBlock1} we draw a case where $\pi=\hat\pi_1$. To be more specific,
 if $\hat \pi_{s-1}:= (\ell_{1}^{\, s-1}, \dots,\ell_{N_{s-1}}^{\,s-1})$ and $\hat \pi_s:= (\ell_{1}^{\, s}, \dots,\ell_{N_s}^{\,s})$, 
then either there exists $k\leq N_{s-1}^{\, s-1}$ such that $ \pi=(\ell_{k}^{\, s-1}, \dots,\ell_{N_{s-1}}^{\,s-1}, \ell_{1}^{\, s}, \dots,\ell_{N_s}^{\,s})$, or $\pi=\hat\pi_s$ (and thus $|\ell_1^{s}|\leq L^{1/4}$). 
The choice of $\pi$ could be not unique. 
To overstep this problem, among all the possible candidates for $\pi$, we choose the one with the minor number of stretches
 which contains $\hat\pi_1$.
Therefore we replace
$\pi$ by $-\pi:= (-\ell_{k}^{\, s-1}, \dots,-\ell_{N_s}^{\,s})$ inside $\hat\pi_{s-1}\cup \hat\pi_s$. It is easy to check that 
after this last transformation, $\hat \pi_s$ achieves the upper exit condition.
However, after this transformation it could be necessary to slightly redefine $\hat \pi_{s-1}$ and $\hat \pi_{s}$ in order to obtain two proper oriented 
 blocks, say $\hat \pi_{s-1}'$ and $\hat \pi_{s}'$,
as pictured in Figure \ref{fig:FlipBlock1}. 
A very crude bound tells us that the number
of ancestors of a macro-block by this last transformation 
is bounded above by its total number of stretches, which is smaller than $l$. 


The procedure described above corresponds to the application 
$\mathcal A_l$ taking as an argument any $\hat \lambda\in \Omega_{l}^x$ such that the last block of $\hat \lambda$
satisfies either the \emph{upper  exit condition} or the \emph{lower exit condition} and maps it onto 
some $\hat \lambda_{\tx{NE}}\in \Omega_l^{\tx{NE}}$. We conclude that, for every $\gamma\in  \Omega_l^{\tx{NE}}$, we have 
\begin{equation}\label{endup3}
\Big| (\mathcal{A}_{l})^{-1}(\gamma)\Big |\leq  8l.
\end{equation}

\subsubsection{Macro-block concatenating procedure}
We consider a given $\hat \Lambda=(\hat\Lambda_1,\dots\hat\Lambda_m)\in \mathcal W_{3,L}$ and we recall \eqref{defll} so that 
$\hat \Lambda\in \Omega_{l_1}^{x_1}\times\dots\times\Omega_{l_m}^{x_m}$. 
At this stage, it is crucial to understand why, except maybe for $j=m$, all non empty macro-blocks $\hat \Lambda_j$ from $\hat \Lambda$
have a last oriented block that satisfies either the \emph{upper  exit condition} or the \emph{lower exit condition}.
To this purpose we consider $\Lambda_j=(\pi_1,\dots,\pi_{r_j})$ the  ancestor of $\hat \Lambda_j=(\hat \pi_1,\dots,\hat \pi_{\hat r_j})$ by $\psi_L^2\,  o \, \psi_L^1$. There are two alternatives at this stage: either the \emph{large stretch removing procedure} in Step 1 has completely removed $\pi_{r_j}$ 
 and then  $\hat \pi_{\hat r_j}$ is associated with one of the $(\pi_k)_{k\leq r_j-1}$ which all satisfy either the \emph{upper exit condition} or the \emph{lower exit condition}, or $\hat \pi_{\hat r_j}$ is associated with $\pi_{r_j}$. In this last case, we recall that the very last stretch of $\pi_{r_i}$  (which is also the last stretch of  $ \Lambda_j$) must cross all the macro-block so that a new macro-block with a different orientation can start (see Figure \ref{fig:IPDRW9} or Figure \ref{fig:FlipBlock1}
 ). This last condition,  depending on the orientation of $\Lambda_i$, implies 
 that $\pi_{r_j}$ also satisfies either the \emph{upper exit condition} or the \emph{lower exit condition} and so do 
 $\hat \pi_{\hat r_j}$.
 
We are now ready to define $\psi_L^3$.  We begin with deleting the empty macro-blocks 
in $\hat \Lambda$, so that it becomes $(\hat \Lambda_{i_1},\dots, \hat \Lambda_{i_{\overline m}})\in \Omega_{l_{i_1}}^{x_{i_1}}\times\dots\times\Omega_{l_{i_{\overline m}}}^{x_{i_{\overline m}}}$, where  $(l_{i_1},\dots,l_{i_{\overline m}})$ is the subsequence of $(l_1,\dots,l_m)$ containing only its non-zero elements. Then we set 
\begin{equation}\label{trhuu}
\overline \Lambda= \big(\overline \Lambda_{i_1},\dots,\overline \Lambda_{i_{\overline{m}}}\big):=\big( \mathcal A_{l_{i_1}}(\hat  \Lambda_{i_1}),\dots, \mathcal A_{l_{i_{\overline m}}}(\hat \Lambda_{i_{\overline m}})\big)\in \Omega_{l_{i_1}}^{\tx{NE}}\times\dots\times\Omega_{l_{i_{\overline m}}}^{\tx{NE}}
\end{equation}
and we let $\psi_L^3(\hat \Lambda)$ be the two-sided path obtained by concatenating all the macro-blocks in  $\overline \Lambda$, i.e.,  
\begin{equation}\label{psiL3}
\psi_L^3(\hat \Lambda)= \overline \Lambda_{i_1}\oplus \dots \oplus \overline \Lambda_{i_{\overline{m}}}.
\end{equation}
As a result, the image set of $\mathcal W_{3,L}$ by $\psi_L^3$ is denoted by $\mathcal W_{4,L}$ and it is a  subset of $\cup_{n=1}^{L} \Omega_n^{\tx{NE}}$. 

 \begin{figure}
 \includegraphics[scale=0.5]{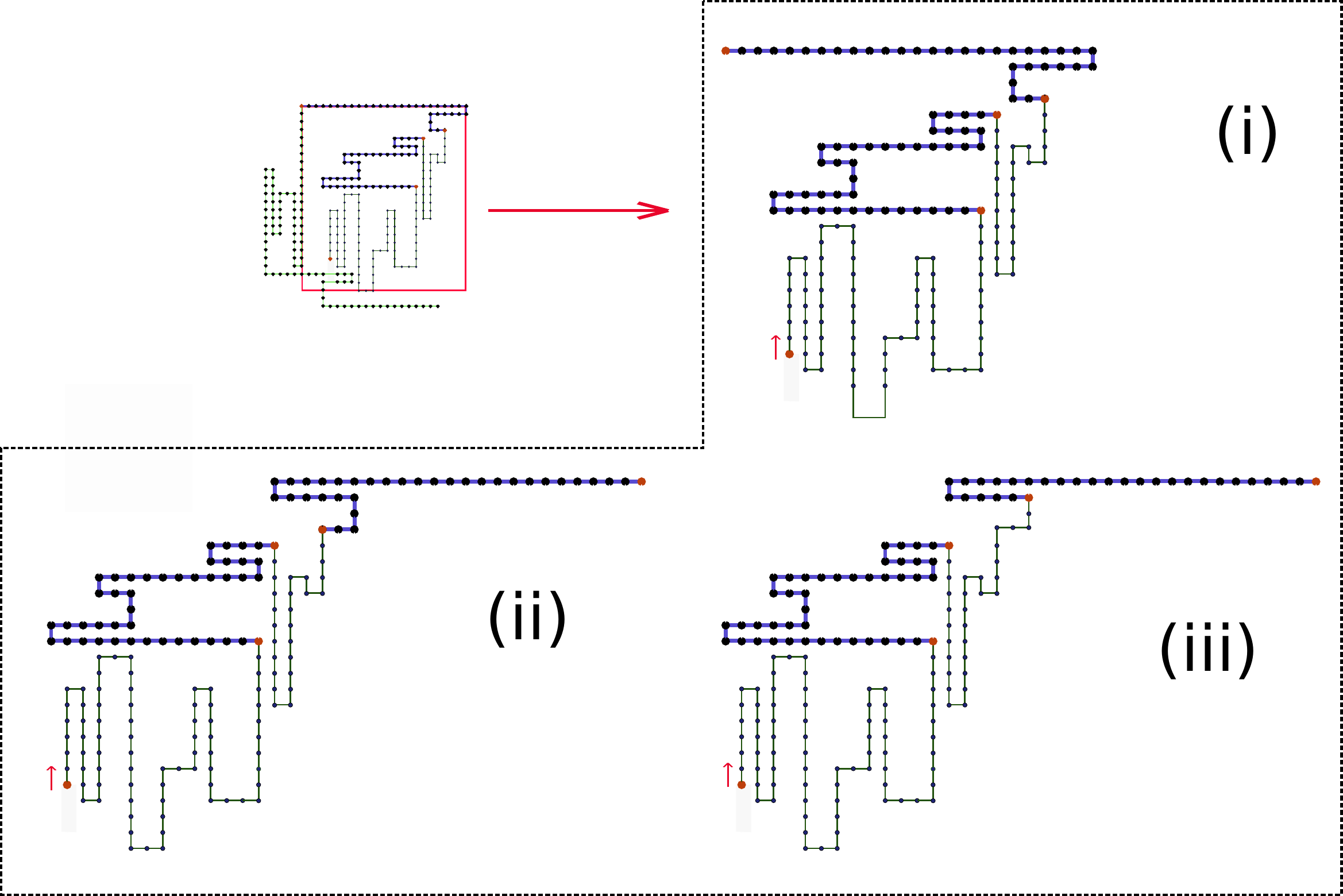}
 \caption{A prudent path obtained by the concatenation of two macro-blocks. 
 We zoom in on the first one, boxed in the rectangle. It has a NE-orientation.
 In $(\texttt i)$ we observe that its last block does not achieves the upper exit condition, but 
 it satisfies the lower exit condition. 
 Therefore, in $(\texttt{ii})$ we apply a spatial symmetry to the last block 
  in such a way that it satisfies the upper exit condition. 
  This changes the structure of the last two blocks. In $(\texttt{iii})$ we redefine the last two blocks.}
 \label{fig:FlipBlock1}
 \end{figure}

The step will be complete once we show that $\psi_L^3$ is sub-exponential.
To that aim, we need an upper bound on the 
 cardinality of $(\psi_L^3)^{-1}(\Gamma)$ that is  uniform on the choice of $\Gamma\in \mathcal{W}_{4,L}$. Thus,
we pick $\Gamma \in \mathcal{W}_{4,L}$, say 
$\Gamma \in \Omega_n^{\tx{NE}}$ with $n\leq L$ and we reconstruct an ancestor $\hat \Lambda$ of $\Gamma$ by $\psi_L^3$. 
We must first choose $m\leq c L^{1/2}$
the number of macro-blocks in $\hat \Lambda$, then choose $\overline m$ the number of non empty blocks in 
$\hat \Lambda$. Then, we must choose the indices of those non-empty macro-blocks which gives us  
$\binom{m}{\overline m}$ possibilities and their lengths $l_{i_1},\dots,l_{i_{\overline m}}$. Once, the latter is done 
it remains to identify the sequence  $(\overline \Lambda_{i_1},\dots,\overline \Lambda_{i_{\overline m}})$ (recall \ref{trhuu})
an we can apply \eqref{endup3} to conclude that the total number of ancestors is bounded above by 
\begin{equation}\label{tgd}
|(\psi_L^3)^{-1}(\Gamma)|\leq \sum_{\overline m\leq m\leq c L^{1/2}}  \sum_{l_{i_1}+\dots+l_{i_{\overline m}}=n}
\binom{m}{\overline m} \prod_{j=1}^{\overline m} 8  l_{i_j} 
\end{equation}
and the r.h.s. in  \eqref{tgd} is smaller than $e^{c_3 L^{1/2} \log L}$ for some $c_3>0$.

\subsection*{Step 4}\label{item:s5} 
In this step we conclude our transformation of the prudent path by showing how we concatenate all stretches picked off by the large stretch removing procedure (cf. Step 1) with the rest of the NE-prudent path provided by Steps 1-3. The result will be a NE-prudent path of length $L$. 

\smallskip

We pick $\Gamma\in \mathcal{W}_{4,L}$, say $\Gamma \in \Omega_n^{\tx{NE}}$ and we denote by $S_{L-n}$
the west-east block of length $L-n$ that maximizes the energy, i.e, $S_{L-n}$ is made of $(L-n)^{1/2}$ vertical stretches of alternating signs of length $(L-n)^{1/2}-1$ each. Then, the image of  $\Gamma$ by $\psi_L^4$ is obtained by concatenating $S_{L-n}$ with $\Gamma$, i.e., 
$$\psi_L^4(\Gamma)=S_{L-n}\oplus \Gamma.$$
The image set of $\mathcal{W}_{4,L}$ by $\psi_L^4$, $\mathcal{W}_{5,L}$, is a subset of $\Omega_L^{\tx{NE}}$ and the number of ancestors of an element in 
$\Omega_L^{\tx{NE}}$ by $\psi_L^4$ is clearly less than $L$, which completes the step.

%
 
%
 %
 
 \medskip
 
 \subsection*{Step 5} 
 
We recall that the composition of those maps $\psi_L^4, \dots, \psi_L^1$ is denoted by $M_L$. In this last step we 
are going to control the energy lost
when we apply $M_L$ to a given $\omega\in \Omega_L^{\tx{PSAW}}$. We aim at showing that 
$H\,(\omega)-H\,(M_L(\omega))=o(L)$  uniformly  on $\omega\in \Omega_L^{\tx{PSAW}}$. 
  
\begin{remark}\label{remint}
 We observe that the image of $\Omega_L^{\tx{PSAW}}$ by  $\psi_L^2 \, \circ\, \psi_L^1$, that is $\mathcal{W}_{3,L}$, contains families 
 of macro-blocks that are a priori not  concatenable. For this reason, we recall \eqref{defll} and  we define the energy of an element
$$\hat \Lambda=(\hat \Lambda_1,\dots,\hat \Lambda_m)\in \Omega_{l_1}^{x_1}\times\dots\times\Omega_{l_m}^{x_m}\in \mathcal{W}_{3,L}$$ as the 
sum of the energies of  its macro-blocks, i.e., 
\begin{equation}
H\,(\hat \Lambda)=\sum_{x=1}^{m} H_{l_x}(\hat \Lambda_x).
\end{equation}
  The sets $\mathcal{W}_{4,L}$ and $\mathcal{W}_{5,L}$, in turn, only contain prudent paths whose  energies are well defined by \eqref{eq:ham1}.
\end{remark}

In part (a) of the proof below we will show that the energy lost when applying $\psi_L^2 \, o\, \psi_L^1$ to a given $\omega\in \Omega_L^{\tx{PSAW}}$
is not larger than $\tilde L+c_1L^{3/4}$ with $c_1>0$ and $\tilde L$ the total length of those stretches removed by the \emph{large stretch removing procedure}.
In part (b) we will show that the mapping $\psi_L^3$ induces at most a loss of energy bounded by $c_2 L^{3/4}$ with $c_2>0$ and finally 
in part (c) we will observe that the gain of energy associated with $\psi_L^4$ is $\tilde L-\tilde L^{1/2}$, which will be sufficient to 
conclude.

\medskip

\begin{enumerate}
\item[(a)]
We pick $\omega\in \Omega_L^{\tx{PSAW}}$ and we denote by 
 $\Lambda=(\Lambda_1,\dots,\Lambda_m)$ its macro-block decomposition. We set
 $\hat \Lambda=(\hat \Lambda_1,\dots,\hat\Lambda_m)= \psi_L^2\, \circ\, \psi_L^1(\Lambda)$. Because of the definition of 
 $H\,(\hat \Lambda)$ in remark \ref{remint},  the interactions between the different macro-blocks of $\Lambda$ 
 do not contribute anymore to the computation of $H\,(\hat \Lambda)$. The next remark allows us to control the sum of the
 interactions between different macro-blocks of $\Lambda$.
 \begin{remark}\label{rthj}
 For $j\in \{1,\dots,m\}$, we let $\ell_1^j$ (resp. $\ell_2^j)$ be the first stretch of the subsequence of odd (resp. even) blocks of $\Lambda_j$.  
 Because of the oriented structure of any macro-block, for every $j=2,\dots, m$, it turns out that
 $\Lambda_j$ interacts with 
 $\Lambda_1\oplus\dots\oplus \Lambda_{j-1}$ only through $\ell_1^j,\, \ell_2^j$ and the number of self-touching between 
 $\Lambda_j$ and  $\Lambda_1\oplus\dots\oplus \Lambda_{j-1}$ is bounded from above by $|\ell_1^j|+|\ell_2^j|$ (see Figure \ref{fig:IPDRW9}).
 \end{remark}
 As a consequence of Remark \ref{rthj}, the energy provided by the interactions between the different macro-blocks of $\Lambda$ is  
 bounded above by $A_1+A_2$ with 
 \begin{align}\label{AA}
 A_1&=\sum_{j=1}^m \left(\, | \ell_1^j|\,  \ind_{\big\{|\ell_1^j|\leq L^{1/4}\big \}}\, +\, |\ell_2^j|\,  \ind_{\big\{|\ell_2^j|\leq L^{1/4}\big \}}\, \right)\\
\nonumber  A_2&=\sum_{j=1}^m \left(\, |\ell_1^j|\,  \ind_{\big\{|\ell_1^j|> L^{1/4}\big \}}\, +\, |\ell_2^j|\,  \ind_{\big\{|\ell_2^j|> L^{1/4}\big \}} \, \right).
 \end{align}
 
 Then, the energy lost during the  transformation of  $\Lambda$ into $\hat \Lambda$ comes on the one hand  from the loss of those interactions between macro-blocks and on the other hand from the energy lost inside every macro-blocks due to the {\it large stretch removing procedure}.
As a consequence, we can write 
\begin{equation}\label{hru}
H\,(\Lambda)-H\,(\hat \Lambda)\leq A_1+A_2+ \sum_{s=1}^m \big(\, H(\Lambda_s)-H(\hat \Lambda_s)\, \big),
\end{equation}
where we recall that for every $s\in \{1,\dots,m\}$, we have $\hat \Lambda_s=\mathcal{R}_{t_s,L}\,\circ\, \mathcal{T}_{t_s,L}(\Lambda_s)$
with $t_s$ the total length of $\Lambda_s$.

At this stage, for $s\in \{1,\dots,m\}$, we need to bound the energy lost in $\Lambda_s$ due to the large stretch removing procedure.
We let $\tilde L_s$ be the total length of those stretches that have been removed and we claim that 
\begin{equation}\label{thsfu}
H(\Lambda_s)-H(\hat \Lambda_s)\leq \tilde L_s-|\ell_1^s|\,   \ind_{\big\{|\ell_1^s|> L^{1/4}\big \}}-|\ell_2^s|\,  \ind_{\big\{|\ell_2^j|> L^{1/4}\big \}}+2L^{1/4}.
\end{equation}
To understand \eqref{thsfu} we must keep in mind that the number of self-touching between two stretches 
is bounded above by the length of the smallest stretch involved. 
This implies that, in the odd subsequence of blocks of $\Lambda_s$, 
the number of self-touching between the first and the second stretch is bounded by the length of the second one.
 Therefore, in the odd subsequence of blocks of $\Lambda_s$, 
 the number of self-touching that are lost when applying the last stretch removing procedure 
is smaller than the sum of all stretches removed in the odd subsequence of oriented blocks minus the length of the very first stretch $\ell_1^s$,
plus the length of the first stretch that has not been removed which, by definition is smaller than $L^{1/4}$. Of course, the same is true for the even subsequence and this explains \eqref{thsfu}.

At this stage, we combine (\ref{AA}\, --\, \ref{thsfu}) and we use the bound $m\leq cL^{1/2}$ (which implies $A_1\leq 2 cL^{3/4}$) to conclude that 
\begin{equation}\label{hru2}
H\,(\Lambda)-H\,(\hat \Lambda)\leq \sum_{s=1}^m \tilde L_s+ 4 c L^{3/4}.
\end{equation}

\item[(b)] 
Note that some energy may also be lost in every macro-block during the third transformation described in Section \ref{reorient}, that is, in the construction of
$\psi_L^3$. Recall \eqref{trhuu}  and the fact that the image of $\hat \Lambda$ by $\psi_L^3$ is denoted by 
$\overline \Lambda$ and has a macro-block decomposition denoted by $(\overline \Lambda_{i_1},\dots,\overline \Lambda_{i_{\overline m}})$. 
Pick $ s\in \{1,\dots,\overline m\}$ and note that after the first two transformations described in Section \ref{reorient}, the macro-block $\hat \Lambda_{i_s}$
has a north-east orientation. 
In case the very last macro-block of $\hat \Lambda_{i_s}$ already  satisfies the upper exit condition, then the third transformation does nothing and $\hat \Lambda_{i_s}=\overline  \Lambda_{i_s}$. 
In case the very last macro-block of $\hat \Lambda_{i_s}$  
satisfies the lower  exit condition,  we observe that 
 it means necessarily that  the {large stretch removing procedure} 
has not removed completely the very last block  of $\Lambda_{i_s}$. 
Therefore, we apply the third transformation that changes the sign of every stretches in the last block and, if its first stretch is larger
 than $L^{1/4}$, then the third transformation also changes the sign of the stretches of $\hat \Lambda_{i_{s-1}}$ between its last
  stretch smaller than $L^{1/4}$ and its very last stretch.
   The existence of such stretch is ensured by the {large stretch removing procedure} 
that we applied to the very last block  of $\Lambda_{i_s}$, as we discussed in. Section \ref{reorient}.
Therefore, by definition, in the third transformation we have lost at most $L^{1/4}$ contacts and   
consequently  
\begin{equation}\label{stecko}
H(\hat \Lambda)-H(\overline \Lambda)\leq \sum_{s=1}^{\overline m} \big(\, H(\hat \Lambda_{i_s})-H(\overline \Lambda_{i_s})\, \big) \leq \overline m L^{1/4}\leq cL^{3/4}.
\end{equation}

\item[(c)]
With the help of \eqref{hru} and \eqref{stecko} above we have proven that for every $\Lambda\in \Omega_L^{\tx{PSAW}}$, by letting $\overline 
\Lambda$ be the image of $\Lambda$ by $\psi_L^3\, \circ\, \psi_L^2 \, \circ\, \psi_L^1$, it holds that
\begin{equation}\label{stecko2}
H(\Lambda)-H(\overline \Lambda)\leq \sum_{s=1}^{m} \tilde L_s+ 5 c L^{3/4}.
\end{equation}
For notational convenience we set $\tilde L:=\sum_{s=1}^{m} \tilde L_s$.
In Step 4, we have built $M_L(\Lambda)$  by concatenating a square block of length $\tilde L$ with $\overline \Lambda$. The 
interactions inside the large square block are $\tilde L-2 \tilde L^{1/2}$ and therefore 
\begin{equation}\label{stecko3}
H(M_L(\Lambda))\geq \tilde L-2 \tilde L^{1/2}+H(\overline \Lambda).
\end{equation}
Finally, (\ref{stecko2}\,--\, \ref{stecko3}) imply that for every $L\in \N$ and every $\Lambda\in \Omega_L^{\tx{PSAW}}$,
\begin{equation}\label{stecko4}
H(\Lambda)-H(M_L(\Lambda))\leq  2 \tilde L^{1/2}+ 5c L^{3/4}\leq  2 L^{1/2}+ 5c L^{3/4},
\end{equation}
and this completes the proof.


\end{enumerate}

 \section{Proof of Theorem \ref{Thm3} }
\label{proofThm3}
We pick $L\in \N$ and we consider $\mathcal{S}_L$ the partially directed path that maximizes the self-touching number. 
We have  already seen in Step 4 of Section \ref{proof of thh1} that $\mathcal{S}_L$  is made of $\sqrt{L}-1$ vertical stretches of length $\sqrt{L}$ each and  that 
$H(\mathcal{S}_L)=L-2\sqrt{L}$. Our proof goes as follows: for every $\epsilon\in (0, 1/60)$ we build the set of path 
$\mathcal{G}_{\epsilon,L}\subset \Omega^{\tx{ISAW}}_L$ such that   for every 
$L$ and $\epsilon$ 
\begin{enumerate}
\item $H(\pi)= H(\mathcal{S}_L)-13 \epsilon L$, for every $\pi \in \mathcal{G}_{\epsilon,L}$,
\item $|\mathcal{G}_{\epsilon,L}|=\binom{ L/60}{\epsilon L}.$
\end{enumerate}
As  a consequence
\begin{align}\label{tautauju}
\nonumber F^{\tx{ISAW}}(\beta)&:=\liminf_{L\to \infty} \frac{1}{L} \log Z_{\beta,L}^{\tx{ISAW}}\geq \sup_{\epsilon>0}\bigg\{\lim_{L\to \infty} \frac{1}{L} \log \binom{L/60 L}{\epsilon L} + \frac{\beta}{L} \big( H(\mathcal{S}_L)-13\epsilon L\big)\bigg\}\\
&\geq \beta+  \sup_{\epsilon>0}\bigg\{\lim_{L\to \infty} \frac{1}{L} \log \binom{L/60}{\epsilon L} -13\,  \beta\,  \epsilon \bigg\},
\end{align}
and this completes the proof since the supremum of the r.h.s. in \eqref{tautauju} is strictly positive because of 
our choice of $\epsilon$.

\begin{figure}
\includegraphics[scale=0.5]{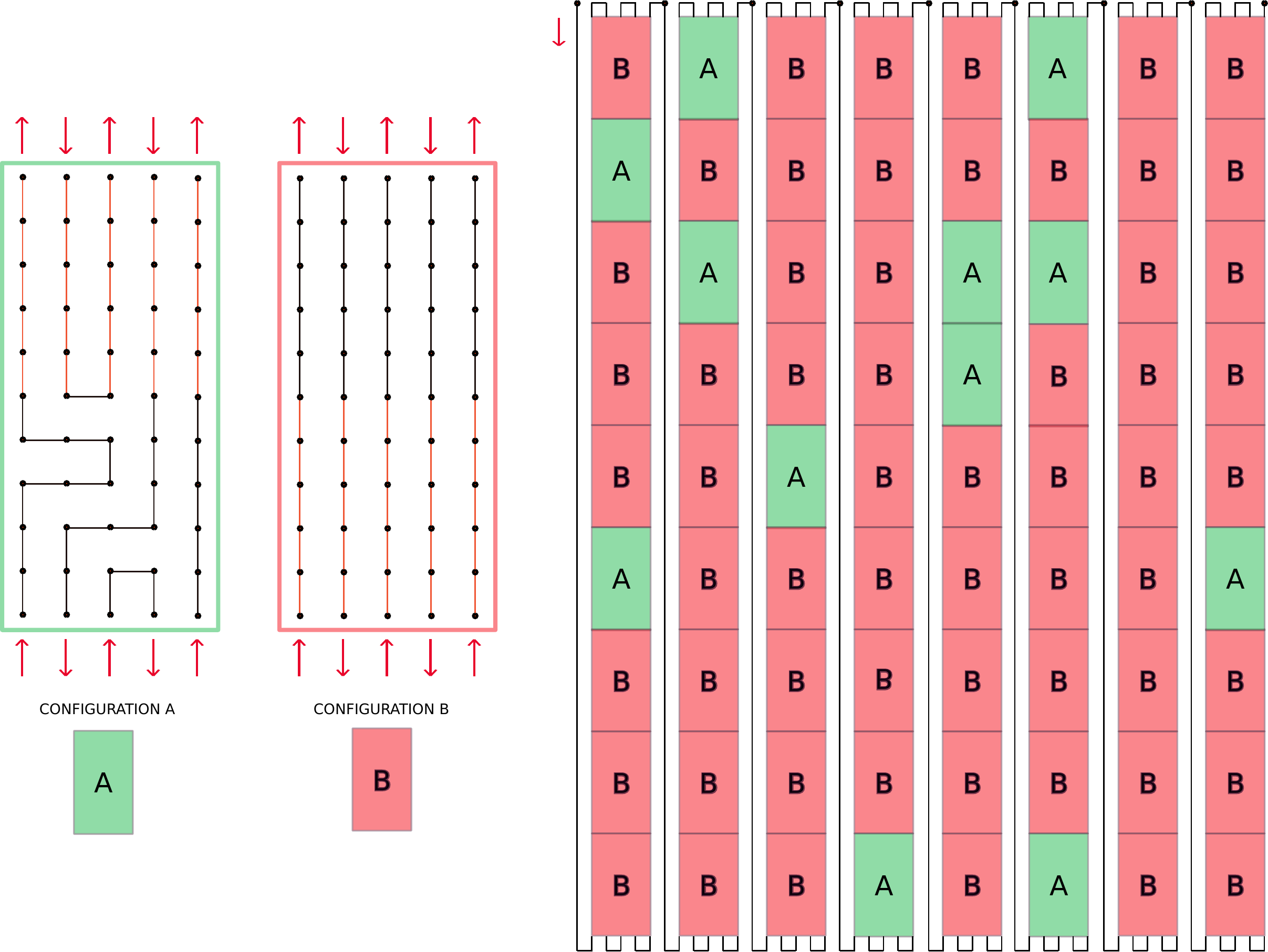}
\caption{On the left configuration $A$ and  $B$  are drawn. The big squared block of size $\sqrt{L}$ on the right is subdivided into 
$L/60$ rectangular boxes, each of them can be filled with configuration $A$ or $B$ without changing the fact that the resulting path 
is self-avoiding. The set $\mathcal{G}_{\epsilon, L}$ contains all path obtained by filling $\epsilon L$ boxes with configuration $A$ and 
the all the others with configuration $B$.\\
Note, in the picture you have to run over the path by starting on the left top, following the 
direction given by the arrow.
This forces to cross any configuration $A$ and $B$ in a unique way, marked by the arrow on the left side of the 
picture.}
 \label{fig:ISAWB}
\end{figure}

It remains to build the sets $\mathcal{G}_{\epsilon,L}$. First, we partition the collections of $\sqrt{L}-1$
vertical stretches of $\mathcal{S}_L$ into groups of $6$ consecutive vertical stretches and then each group 
is divided vertically into rectangles of heights $10$. This gives us a total of $L/60$ rectangular boxes. On the left hand side of  Figure \ref{fig:ISAWB} two configurations (denoted by $A$ and $B$) are drawn and  each of them is  made of $60$ steps. An important feature of configurations $A$ and $B$ is that one can fill every rectangular box with an $A$ or with a $B$ configuration (see the r.h.s. of Figure \ref{fig:ISAWB}) and still recover a self-avoiding path of size $L$.   
The $\mathcal{S}_L$ path is obtained by filling all boxes with configuration $B$. We also note that filling a box with an
$A$ configuration provides exactly $13$ self-touching less than filling the same box with a $B$ configuration. 

The  set $\mathcal{G}_{\epsilon,L}$ contains all paths obtained by filling the $L/60$ boxes with $\epsilon L$ blocks of type $A$ 
and $L(\frac{1}{60 }-\epsilon)$ blocks of type $B$. Thus, the cardinality of $\mathcal{G}_{\epsilon,L}$ is $\binom{L/60}{\epsilon L}$
and the Hamiltonian of every path in $\mathcal{G}_{\epsilon,L}$ is equal to  $H(\mathcal{S}_L)-13 \epsilon L$. This completes the proof.


\section{Free Energy: convergence in the right hand side of \eqref{eq:FreeEnergyIPSAW} }
\label{app:freeenergy}
The goal of this section is to prove the existence of the free energy for the NE-prudent walk. 
For this purpose, we aim at using a super-additive argument, cf. Proposition A.12 in \cite{GB07}.
It turns out that the sequence $\left(\mathrm Z_{\beta,\, L}^{\texttt{NE}}\right)_{L\in\mathbb N}$ is not $\log$ super-additive, therefore we introduce a super-additive process, for which the free energy exists, and we show that it rounds up/down $\mathrm Z_{\beta,\, L}^{\texttt{NE}}$.


\smallskip

The energy associated with a path is described by an Hamiltonian function $\mathrm H\,(w)$, cf. \eqref{eq:ham1}. 
We let  $\Omega_L^{\mathtt{NE},*}\subseteq \Omega_L^{\mathtt{NE}}$ be the set of the whole NE-prudent paths for which the upper
exit condition is satisfied by all the blocks of the path and we let $\tilde\Omega_L^{\mathtt{NE},*}\subseteq \Omega_L^{\mathtt{NE},*}$ be the set of the $\mathtt{NE}$-prudent paths in 
$\Omega_L^{\mathtt{NE},*}$ for which the first stretch of the path is equal to $0$. 
We let 
$\mathrm Z_{\beta,\, L}^{\texttt{NE}, *}$ and $\tilde{\mathrm Z}_{\beta,\, L}^{\texttt{NE},*}$
be the partition functions associated with these sets respectively.
In the next lemma we prove that $\tilde{\mathrm Z}_{\beta,\, L}^{\texttt{NE},*}$ is $\log$ super-additive.
\begin{lemma}
\label{lemma:SA*0}
The sequence $\left(\tilde{\mathrm Z}_{\beta,\, L}^{\texttt{NE},*}\right)_{L\in \mathbb N}$ is $\log$ super-additive. As a consequence, the free energy $\tilde{\mathrm F}^{\mathtt{NE},*}(\beta)$ exists ant it is finite, i.e.,
$$
\tilde{\mathrm F}^{\mathtt{NE},*}(\beta):=\lim_{L\to\infty}\frac{1}{L}\log \tilde{\mathrm Z}_{\beta,\, L}^{\texttt{NE},*} =\sup_{L\in\mathbb N} \frac{1}{L}\log \tilde{\mathrm Z}_{\beta,\, L}^{\texttt{NE},*}<\infty.
$$
\end{lemma}
\begin{proof}
We start by showing the super-additivity. We pick $0\leq L'\leq L$ and we consider two paths $w_1\in \tilde\Omega_{L'}^{\mathtt{NE},*}$ and $w_2\in \tilde\Omega_{L-L'}^{\mathtt{NE},*}$. 
We note that we can safely concatenate 
$w_1$ with $w_2$,
by obtaining the path  $w_1\oplus w_2$, which is an element of $\tilde\Omega_L^{\mathtt{NE},*}$. Moreover, 
we note that $\mathrm H\,(w_1\oplus w_2)\geq \mathrm H_{L'}(w_1)+\mathrm H_{L-L'}(w_2)$. 
We conclude that,
\begin{equation}
\begin{split}
&\tilde{\mathrm Z}_{\beta,\, L}^{\texttt{NE},*}
\geq 
\sum_{\substack{w=w_1\oplus w_2\, : \\ (w_1,w_2)\in \tilde\Omega_{L'}^{\mathtt{NE},*}\times\,  \tilde\Omega_{L-L'}^{\mathtt{NE},*}}}e^{\,\beta \mathrm H\,(w)}   \geq
\sum_{(w_1,w_2)\in \tilde\Omega_{L'}^{\mathtt{NE},*}\times\,  \tilde\Omega_{L-L'}^{\mathtt{NE},*}}
e^{\,\beta \mathrm H_{L'}(w_1)}e^{\,\beta \mathrm H_{L-L'}(w_2)} =\tilde{\mathrm Z}_{\beta,\, L'}^{\texttt{NE},*}\tilde{\mathrm Z}_{\beta,\, L-L'}^{\texttt{NE},*}.
\end{split}
\end{equation}
To prove that the limit is finite, we observe that $\mathrm H\,(w)\leq L$ and thus
$\tilde{\mathrm Z}_{\beta,\, L}^{\texttt{NE},*} \leq e^{\beta L}\, |\tilde\Omega_L^{\mathtt{NE},*}|\leq e^{\beta L}\, |\Omega_L^{\mathtt{NE}}|$. 
This conclude the proof because $\limsup_{L\to\infty}\frac{1}{L}\log |\Omega_L^{\mathtt{NE}}|<\infty$.
\end{proof}
We are going to compare $\tilde{\mathrm Z}_{\beta,\, L}^{\texttt{NE},*}$ with ${\mathrm Z}_{\beta,\, L}^{\texttt{NE},*}$, in order to obtain the existence of the free energy for ${\mathrm Z}_{\beta,\, L}^{\texttt{NE},*}$. 
By definition it holds that $\tilde{\mathrm Z}_{\beta,\, L}^{\texttt{NE},*}\leq {\mathrm Z}_{\beta,\, L}^{\texttt{NE},*}$.
On the other hand, we observe that given $w\in \tilde\Omega_L^{\tx{NE}}$, if we keep out
 the first stretch of $w$ (which has $0$ length), 
 then we obtain a path $w'\in \Omega_{L-1}^{\tx{NE}}$. The map which associates $w$ with $w'$  is bijective, because there is only one way to add a stretch of $0$ length to a block. Since $\mathrm H\,(w)=\mathrm H_{L-1}(w')$, 
we conclude that $\tilde{\mathrm Z}_{\beta,\, L}^{\texttt{NE},*}\geq {\mathrm Z}_{\beta,\, L-1}^{\texttt{NE},*}$.
As a consequence, we have that
\begin{equation}
\mathrm F^{\mathtt{NE},*}(\beta):=\lim_{L\to\infty}\frac{1}{L}\log {\mathrm Z}_{\beta,\, L}^{\texttt{NE},*}\qquad
\text{and} 
\qquad \mathrm F^{\mathtt{NE},*}(\beta)=\tilde{\mathrm F}^{\mathtt{NE},*}(\beta), \qquad \forall\, \beta \geq 0.
\end{equation}

We are ready to bound from below and above the function ${\mathrm Z}_{\beta,\, L}^{\texttt{NE}}$ by a suitable function for which the free energy exists. 
We let
\begin{equation}
\Phi_{L,\beta}:=\sum_{L'=1}^L {\mathrm Z}_{\beta,\, L'}^{\texttt{NE},*}{\mathrm Z}_{\beta,\, L-L'}^{\tx{IPDSAW}}.
\end{equation}
It is a standard fact, cf.\cite[Lemma 1.8]{GB07} that the existence of the free energy of ${\mathrm Z}_{\beta,\, L}^{\texttt{NE},*}$ and ${\mathrm Z}_{\beta,\, L}^{\tx{IDPSAW}}$ implies the existence of the free energy of $\Phi_{L,\beta}$ and 
\begin{equation}
\label{eq:freenergyconvmax}
\lim_{L\to\infty}\frac{1}{L}\log \Phi_{L,\beta} =\max\Big\{F^{\tx{IPDSAW}}(\beta),F^{\tx{NE},*}(\beta)\Big\}
\end{equation}
where $\mathrm F^{\tx{IPDSAW}}(\beta)$ is the free energy associated with ${\mathrm Z}_{\beta,\, L}^{\texttt{IPDSAW}}$ (its existence was proven in \cite{CGP13}).
\begin{proposition}\label{PropFreenergy}
It holds that 
\begin{equation}
\label{eq:ZNE-ZNE*}
\Phi_{L,\beta}
\leq 
{\mathrm Z}_{\beta,\, L}^{\tx{NE}}\leq 
e^{o(L)}\Phi_{L,\beta}.
\end{equation}
As a consequence we have that the free energy of ${\mathrm Z}_{\beta,\, L}^{\texttt{NE}}$ exists and it is finite, i.e., 
\begin{equation}
\label{eq:freeEnergyNE}
\mathrm F^{\mathtt{NE}}(\beta):=\lim_{L\to\infty}\frac{1}{L}\log {\mathrm Z}_{\beta,\, L}^{\texttt{NE}} <\infty.
\end{equation}
\end{proposition}

\begin{proof}
%
%
To prove the lower bound in \eqref{eq:ZNE-ZNE*} we consider the family of disjoints sets $\Omega_{L'}^{\mathtt{NE},*}\times \Omega_{L-L'}^{\mathtt{PDSAW}}$, with $L'\in \{0,\dots, L\}$. 
For any $(w',w'') \in \displaystyle\bigcup_{0\leq L'\leq L}\Omega_{L'}^{\mathtt{NE},*}\times \Omega_{L-L'}^{\mathtt{PDSAW}}$.
Let $w=w'\oplus w''$ be the concatenation of $w''$ with $w'$. Since $\mathrm H_N(w)\geq \mathrm H_{L'}(w')+\mathrm H_{L-L'}(w'')$ we have
\begin{align*}
&\mathrm Z_{L,\beta}^{\mathtt{NE}} := \sum_{w\in \Omega_L^{\tx{NE}}} e^{\, \beta \mathrm H\,(w)} \geq
\sum_{L'=0}^L \sum_{\substack{w\in \Omega_L^{\tx{NE}}\, : \\ w=w'\oplus w'',\\ (w',w'')\in \Omega_{L'}^{\mathtt{NE},*}\times \Omega_{L-L'}^{\mathtt{PDSAW}}}} e^{\, \beta \mathrm H\,(w)} 
\\ & \qquad \geq
 \sum_{L'=0}^L \sum_{\substack{w\in \Omega_L^{\tx{NE}}\, : \\ w=w'\oplus w'',\\ (w',w'')\in \Omega_{L'}^{\mathtt{NE},*}\times \Omega_{L-L'}^{\mathtt{PDSAW}}}} e^{\, \beta (\mathrm H\,(w')+\mathrm H\,(w'))}=
 \sum_{L'=1}^L {\mathrm Z}_{\beta,\, L'}^{\texttt{NE},*}{\mathrm Z}_{\beta,\, L-L'}^{\texttt{IPSAW}} .
\end{align*}

The strategy to prove the upper bound in \eqref{eq:ZNE-ZNE*} is similar to the strategy used for the proof of Theorem \ref{thh1} in Section \ref{proof of thh1}. 
%
%
%
%
%
To be more precise, we associate with each $w\in \Omega_L^{\mathtt{NE}}$ two paths $u'\in \Omega_{L'}^{\mathtt{NE},*}$ and  $w'\in \Omega_{L-L'}^{\mathtt{PDSAW}}$, for some $0< L'< L$, with $L'=L'(w)$, through a sub-exponential function (cf. Definition \ref{subex}).
We let 
$(\pi_1,\dots,\pi_r)$ be the block 
decomposition of $w$. 
We consider the last block $\pi_r$, of length $L-L'$, for some $L'<L$. We apply the large stretch removing procedure to $\pi_r$, i.e., by starting from the first stretch,
we pick off all the consecutive stretches larger than $L^{1/4}$ in the block $\pi_r$. Let $\pi_r'$ be the result of this operation. 
Let $\tilde L$ be the total length of the stretches that we picked off. We define an oriented block made of $\sqrt{\tilde L}$ vertical stretches of alternating sings of length $\sqrt{\tilde L}-1$. 
This configuration maximizes the energy of a block of length $\tilde L$. 
The orientation of this block is the same
as that of $\pi_r$.
We concatenate this block with $\pi_r'$ and we call
$w'$ the path obtained at the end of this operation. 
By construction $w'\in \Omega_{L-L'}^{\mathtt{PDSAW}}$. 
We let $u':=\pi_1\oplus \dots\oplus \pi_{r-1}$, so that  $(u',w')\in \Omega_{L'}^{\mathtt{NE},*}\times \Omega_{L-L'}^{\mathtt{PDSAW}}$.
%
The computations we did in Steps $1-4$ in Section \ref{proof of thh1} ensure that the function which associates $w$ with $(u',w')$ is sub-exponential and, by reasoning as in Step $5$ of Section \ref{proof of thh1}, it turns out that $H\,(w)- \big(H_{L'}(u')+H_{L-L'}(w')\big)\leq o(L)$, uniformly on $w\in \Omega_L^{\tx{NE}}$. This suffices to conclude the proof.

\end{proof}

\section*{Acknowledgements}

 We thank Philippe Carmona for fruitful discussions.
\bibliographystyle{apalike}
\bibliography{biblioT.bib}

\end{document}